\documentclass{article} 

\usepackage{hyperref}
\usepackage{arxiv}
\usepackage{float}

\usepackage[english]{babel}		

\usepackage{polski}
\usepackage[LY1]{fontenc}
\usepackage[utf8]{inputenc}

\usepackage{graphicx}
\usepackage{soul}

\usepackage{amsmath,amsfonts,amsthm}										
\usepackage{fancyhdr}
\numberwithin{equation}{section}															
\numberwithin{figure}{section}																
\numberwithin{table}{section}																
\usepackage{upgreek}
\usepackage{ytableau}
\usepackage{color}

\newcommand{\ck}{\color{black}}

\newcommand{\N}{\mathbb N}

\newcommand{\R}{\mathbb R}

\newcommand{\rank}{\mbox{rank}}

\newcommand{\ord}{\mathrm{ord}}

\newtheorem{theorem}{Theorem}
\newtheorem{proposition}[theorem]{Proposition}

\newtheorem{corollary}[theorem]{Corollary}
\newtheorem{lemma}[theorem]{Lemma}

\newtheorem{definition}[theorem]{Definition}

\usepackage{amsbsy}

\long\def\llap #1{\hbox to 0pt{\hss #1}}

\newenvironment{example}[1][]{\refstepcounter{theorem}\par\medskip
   \textbf{Example~\thetheorem. #1} \rmfamily}{\medskip}

\newenvironment{remark}[1][]{\refstepcounter{theorem}\par\medskip
   \textbf{Remark~\thetheorem. #1} \rmfamily}{\medskip}

\newenvironment{algorithm}{
\bigskip\noindent
\refstepcounter{theorem}
\textbf{Algorithm \thetheorem.}}

\numberwithin{theorem}{section}

\title{Local Integrable Symmetries of Diffieties}

\author{Yirmeyahu J. Kaminski\thanks{School of Mathematical Sciences, Holon Institute of Technology, Holon, Israel, email: kaminsj@hit.ac.il} \, and Fran\c{c}ois Ollivier\thanks{LIX, CNRS, École polytechnique, 91128 Palaiseau Cedex, France; with the support of:\hfill\break ERC project ODELIX, ``Solving differential equations fast, precisely, and reliably'', ERC-2023-ADG grant, number 101142171;\hfill\break ANR Project NODE, ``Résolution numérique-symbolique d'équations différentielles'', Projet-ANR-22-CE48-0016 and\hfill\break ANR project OCCAM, ``Théorie et pratique de l'élimination différentielle'', Projet-ANR-22-CE48-0008}}

\date{}

\begin{document}

\maketitle

\begin{abstract}
In the framework of diffieties, introduced by Vinogradov~\cite{Vinogradov1986}, we introduce integrable infinitesimal symmetries and show that they define a one parameter pseudogroup of local diffiety morphisms. We prove some preliminary results allowing to reduce the computation of integrable infinitesimal symmetries of a given order to solving a system of partial differential equations. We provide examples for which we can reduce to a linear system that can be solved by hand computation, and investigate some consequences for the local classification of diffiety, with a special interest for testing if a diffiety is flat~\cite{Fliess1999}.
\end{abstract}

\section*{Introduction}

The notion of diffiety, introduced by Vinogradov~\cite{Vinogradov1986}, offers a convenient framework for the study of differential equations. It allows a precise description of many notions in control theory, such as differential flatness, introduced by Fliess, Lévine, Martin and Rouchon~\cite{Fliess1999}. If no effective necessary and sufficient condition of flatness are known, some criteria are available for special cases. For example, a criterion due to Cartan~\cite{Cartan1915} provides a test for flatness of driftless systems with two controls~\cite{Rouchon1993}. A recent work deals with the case of systems having $n$ state variables and $n-1$ commands~\cite{KLO-20}. 

But little is known about the classification of diffieties, up to isomorphism. Contact varieties of dimension $3$ define driftless systems with two control, such as the classical car model~\cite{Kaminski2018}. They are all locally isomorphic, which is basically equivalent to the fact that all controllable driftless systems with two controls and three states are flat, but their global geometry is more complicated~\cite{Geiges2001}. One may also mention the more complicated problem of classifying sub-Riemannian varieties~\cite{Beliavsky1999}.

Flatness is closely related to Monge problem~\cite{Monge1784} that seems to have been the first to consider examples of differential systems, that are flat according to our modern definition. But Monge does not specify an independent variable, so that ``changes of time'' are allowed and his standpoint is closer to ``orbital flatness''~\cite[4.4]{Fliess1997b}. This is also the classical standpoint of diffiety theory.

We try here to initiate a study of the local symmetries of diffieties, with a single Cartan field. For this, we look for one parameter pseudogroups of automorphisms, generated by \emph{integrable} vector fields.
\bigskip

Our approach should not be confused with the classical theory inspired by Lie that aims at deciding equivalence of differential equation, under the action of a given group, looking for differential invariants. One may quote the well known work of Tresse on differential equations of order $2$, the contributions of Cartan~\cite{Cartan1937}, Olver~\cite[chap.~10]{Olver1995} or the work of Neut and Petitot~\cite{Neut2002} providing a full implementation of Cartan's method, allowing to solve order $3$. 

This is in some way easier, as it focuses mostly on the finite dimensional case or cases that could be easily reduced to working with jets of bounded order, and more complicated as only special cases of groups are considered. In our setting, the finite dimensional case is always trivial (see subsec.~\ref{subsec::not_stringly_access}).
When underdetermined systems are not considered, the geometrical objects associated to ordinary differential equations are in general finite dimensional: for example, all fields are integrable. 

One may notice that Zharinov~\cite[chap.~3]{Zharinov1992} introduces infinitesimal symmetries without considering integrability issues, that is the possibility to associate to a vector field a one parameter pseudogroup, which is one of the main concerns of this paper, but is of no importance for his approach.

A main concern for such works is to be able to test involutivity of a differential system and to be able to reduce any system to an equivalent involutive system. Roughly speaking, this means that one needs to decide up to what order should the system equations be differentiated to get all the differential relations of a given order satisfied by the solutions~\cite{Malgrange2005}, according to the Cartan--Kuranishi theory. ``integrable''.

One may often refer to systems possessing such properties of completeness as ``integrable''. Moreover, some classical mathematical tools considered in such approaches, such as Spencer cohomology~\cite{Malgrange1966}, do not seem to provide much help for deciding if diffieties are locally isomorphic. At best, one may recover elementary conditions, such as the equality of the differential dimensions, or of the orders for differential dimension $0$.  Some attempts to test flatness using those tools have not been successful up to now. 

We were unable to find in the literature a notion that could fully match our definitions. The notion of integrability considered by Pollack~\cite{Pollack1974} is of a different nature, as one considers pseudogroups acting on finite dimensional spaces, such as isometries of the Euclidean space, or homeomorphisms of $\mathbb{R}^n$. Even when they consider the infinite dimensional geometry of jet space, works inspired by Cartan and his equivalence method, such as Olver~\cite[chap.~2]{Olver1986} start with an explicitly given group action. Or, following E.~Noether~\cite{Kosmann-Schwarzbach2011}, one may study fields that involve higher order derivatives in their expression, the ``higher-order'' or ``generalized'' symmetries~\cite[chap.~5]{Olver1986}, but that cannot define a local group action on the whole diffiety. 

Our problem is to consider fields acting on jets of order $r$, that can depend on jets of greater order, so generalized symmetries, but that can still define a local group action. This would correspond to Cartan's ``équivalence absolue'' Cartan~\cite{Cartan1914}. One may also compare our work with Clelland~\textit{et al.}~\cite{Clelland2024} who investigate dynamic feedback linearizability of control systems with symmetries. But again, they restrict their approach to symmetries associated with transformations of the state space. For many examples, we also need to reduce to this case, but we are then able to prove that no other solution exists.
\bigskip

Section~\ref{sec::framework} introduces our theoretical framework, using a variant of Vinogradov's diffieties. We define diffiety extensions, flat extension and the linearized system, showing that the linearized system associated to a controlable diffiety is a flat extension, which allows to easily construct a basis of it.

In section~\ref{sec::one_parameter}, we introduce integrable infinitesimal symmetries and the one parameter local automorphisms they define. Th.~\ref{thm::rouchon_like} provides an analog of the Sluis -- Rouchon theorem, allowing to prove that integrable symmetries for system that do not satisfy this flatness condition must be of order $0$. In subsection~\ref{subsec::PDE}, we investigate the PDE structure obtained by considering both the Cartan field $\uptau$ and a symmetry $\updelta$. We define then a notion of characteristic set for analytic systems, extending the classical notion of differential algebra~\cite{Kolchin1973}, and reducing testing integrability to a change of ordering. We are moreover able to bound the order in $\updelta$ for such computations. The integrable symmetries are defined by algebraic PDE that themselves define diffieties.

Sections~\ref{sec::single} and \ref{sec::multi} are respectively devoted to the computation of some examples, respectively in differential dimension $1$, \textit{i.e.}\ one control, which is the easiest case, or greater dimension. We show that integrable symmetries are reduced to $0$ for generic systems of a degree great enough.

The last section~\ref{sec::classification} investigates the use of the local integrable symmetries for the local classification of diffieties. Basically, one encounters diffieties with a reduced set of integrable symmetries, such as a finite dimensional vector space, or at the other end the flat diffiety (and all diffiety being the product of a flat diffiety and some other one) that admits a infinite dimensional set of integrable symmetries, characterized by a PDE system and that we may describe using the differential transcendence function of its prime components, some having a degree at least proportional to the product of their differential dimension and the order of the symmetries.
\bigskip

This paper involves frequent references to notions that come from control theory. Although possible developments could obviously help designing genuine applications, we need to stress that the results presented here have no direct applicability and that control problems are mostly for us at this stage a valuable source of inspiration and mathematical intuition.

\section{Theoretical framework}
\label{sec::framework}

\subsection{Diffieties}\label{subsec::diffieties}

We will work in the framework of diffiety theory, introduced by Vinogradov~\cite{Vinogradov1986,Zharinov1992}.

We need first a technical definition, following~\cite{FO-BS-07,Ollivier2023}.
\begin{definition}\label{def::order}
    We say that a function $F$ depending of $x_1$, \dots, $x_n$ and a finite number of their derivatives is of order $k$ in $x_j$ at point $x=\eta$ if $F$ does not depend on any derivative of $x_j$ greater that $k$ and $\partial F/\partial x_j^{(k)}(\eta)\neq0$. We write then $\ord_{\eta,x_j}F=k$. We say that $F$ is of order $k$ at point $x=\eta$ if $k=\max_{j=1}^{n}\ord_{\eta,x_j} F$ and we write then $\ord_{\eta}F=k$.
\end{definition}

\begin{definition}\label{def::diffiety}
A diffiety is a real variety $\mathfrak{V}$ of denumerable dimension, equipped with a vector field $\uptau$, the Cartan field of the diffiety. We may have also to consider partial diffiety equipped with a $r$-uple $\Delta$ of vector fields.

The ring of functions $\mathcal{O}(\mathfrak{V})$ is the ring of $\mathcal{C}^{\infty}$ functions that depend only on a finite number of coordinates. The topology on the diffiety is the coarsest topology such that all coordinate functions are continuous.

A morphism of diffiety $\phi:(\mathfrak{V}_1,\Delta_1)\mapsto(\mathfrak{V}_2,\Delta_2)$, with $\sharp\Delta_1=\sharp\Delta_2=r$ is such that, denoting by $\phi^{\ast}$ the dual application $\mathcal{O}(\mathfrak{V}_2)\mapsto\mathcal{O}(\mathfrak{V}_1)$, we have: $\Delta_{1,i}\circ\phi^{\ast} = \phi^{\ast}\circ\Delta_{2,i}$, for all $i \in [1,r]$. 
\end{definition}
In the sequel, we will only consider ordinary diffieties, \textit{i.e.} with a single vector field $\uptau$, except in subsections~\ref{subsec::PDE}, \ref{subsec::structure} and \ref{subsec::flatness}.

\begin{remark}
    We will freely denote $\uptau x=\dot x$, $\uptau^2 x=\ddot x$, \dots, or $\uptau x=x'$, $\uptau^2 x=x''$, $\uptau^k x= x^{(k)}$, when it is easier for readability.
\end{remark}

Our definition is slightly different from the one introduced by Vinogradov. The main difference is that we consider a single Cartan field, while more generally the classical definition of diffieties relies on an involutive distribution, also called the Cartan distribution. We will also restrict to the analytic case, \textit{i.e.} $V$ is analytic as well as the field $\uptau$ and functions in $\mathcal{O}(\mathfrak{V})$.

\begin{example} The \emph{point diffiety} is $\{0\}$ equipped with derivation $0$.
\end{example}

\begin{example} The \emph{time diffiety} is $\mathbb{R}_t$ equipped with derivation $\partial_t$.

In the same spirit, we may consider the partial derivatives case with $r$ derivations $\R_{t_1, \ldots, t_r}^r$ equipped with partial derivatives $\partial_{t_j}$.
\end{example}

\begin{example}\label{ex::trivial}
The trivial diffiety of differential dimension $n$, denoted by $\mathbb{T}^n$ is $(\R^\mathbb{N})^n$ endowed with the trivial Cartan field:
\begin{equation}\label{eq::trivial}
    \mathrm{d}_{t}:=\sum_{i=1}^n \sum_{\ell \in \N} x_i^{(\ell+1)} \partial_{x_i^{(\ell)}}.
\end{equation}
We can also define a trivial diffiety which partial derivatives. Let $\Delta=\{\updelta_1, \ldots, \updelta_r\}$ and $\Theta$ the free commutative monoid generated by $\Delta$. The trivial diffiety of differential dimension $n$ with $r$ derivatives is $\mathbb{T}_\Delta^n:=(\R^{\N})^n$ equipped with the derivations
\begin{equation}\label{eq::trivial_EDP}
    \updelta_j:=\sum_{i=1}^n \sum_{\theta\in\Theta} \updelta_j\theta x_i \partial_{\theta x_i}.
\end{equation}
\end{example}

\begin{example}\label{ex::jet} The jet space $\mathrm{J}(\mathbb{R},\mathbb{R}^n)$ has a natural structure of diffiety and is isomorphic to $\mathbb{R}_t \times \mathbb{T}^n$ equipped with $\partial_t+\mathrm{d}_{t}$.

In the partial differential case, we identify $\updelta_j$ with $\mathrm{d}_{t_j}$ and the jet space $\mathrm{J}(\mathbb{R}^r,\mathbb{R}^n)$ is isomorphic to $\mathbb{R}_{t_1, \ldots, t_r}^r \times \mathbb{T}_\Delta^n$ equipped with the derivations $\partial_{t_j}+\mathrm{d}_{t_j}+\updelta_j$.
\end{example}

\subsection{Control systems}\label{subsec::control_sys}

Let us consider a control system defined on a chart of a manifold $X$ by the following differential system:
\begin{equation}\label{eq::system}
\dot{x} = f(x,u,t),
\end{equation}
and $\dot t=1$, where $x \in \R^n$, $u \in \R^m$ and $t \in \R$ and $f$ is given by smooth (that is $C^\infty$) functions. 

Notice that while our theoretical framework is the analytic case, and that most results could even be extended to the smooth case. When computational differential algebra is involved for algorithmic purpose, the function $f$ is then assumed to be algebraic. 

This restriction is not so important in practice, since it should be noted that most systems defined by analytic vector fields that are not already algebraic can actually be transformed into algebraic systems when those analytic functions are defined by differential equations. Trigonometric functions can be replaced by their rational parametrization. Exponential and logarithm expressions like $t_1 = \exp(g(x))$ and $t_2 = \ln(g(x))$ can be modeled by adding extra differential equations, as follows: $\dot{t}_1 = \dot{g}(x) t_1$ and $\dot{t}_2 = \frac{\dot{g}(x)}{g(x)}$ and so on. One just has to add new initial conditions in a suitable way, and to be careful that these new system may involve torsion elements and so not be strongly accessible, except when restricted to a leaf of a suitable foliation. See def.~\ref{def::accessibility} and the following comments.

Considering functions defined by ordinary differential equations and
initial conditions, an effective naive implementation of basic operations is straightforward. The main difficulty is to test zero equivalence, which remains a solvable problem. Considering PDE, one
may have more trouble if the initial conditions correspond to a
singular point for some of the equations~\cite{Peladan2002}. More details on such issues would exceed the scope of this paper. For more details, we refer to van der Hoeven's paper~\cite{vanderHoeven2019} and the references therein.

Consequently, the theoretical framework of this paper allows to consider systems defined by transcendental functions, as long as one is able to achieve basic computations with them.

\begin{definition}\label{def::diff_sys}
To this system is associated a diffiety $\mathfrak{X} = \R \times \R^n \times (\R^m)^\N$, endowed with the Cartan field:
\begin{equation}\label{eq::tau}
\uptau := \partial_t + \sum_{i=1}^n f_i(x,u,t) \partial_{x_i} + \sum_{j=1}^m \sum_{k=0}^\infty u_j^{(k+1)} \partial_{u_j^{(k)}}.
\end{equation}
\end{definition}

\begin{remark}
    A simple but important remark is that applying the field $\uptau$ on a function $F$ depending on a finite number of the global coordinates of the diffiety is equivalent to the computation of the time derivative $\dot{F}$, when the coordinates are considered as functions of the time. This equivalence is largely used in the sequel, without further notification and is a justification of the shorthand notation $\dot F$ for $\uptau F$.
\end{remark}

\begin{remark}\label{rem::topology} One may need for this to restrict to an open subset of $\mathfrak{X} = \R \times \R^n \times (\R^m)^\N$ if the functions $f_i$ are not defined everywhere, or if one needs for physical reasons to impose restrictions on further derivatives (bounded acceleration, etc.). Denoting $\mathfrak{X}^{[r]}$ to be the space corresponding to derivatives of order at most $r$ and $\pi_r$ the projection $\pi_r:\mathfrak{X}\mapsto\mathfrak{X}^{[r]}$, for $r$ great enough, we have $\pi_r^{-1}\pi_r\mathfrak{X}=\mathfrak{X}$, according to the topology chosen in def.~\ref{def::diffiety}. Indeed, it is the coarsest topology such that the coordinate functions are continuous, so that an open set is actually defined by its projection on a space of finite dimension.
\end{remark}

\subsection{Diffiety extensions}
\label{subsec::linearized}

Let us remind the notion of diffiety extension, as introduced in~\cite{FO-BS-07}, that will be useful in the sequel. 

\begin{definition}\label{def::diffiety_ext}
We say that a diffiety $\mathfrak{Y}_\uptau$ is an extension of a diffiety $\mathfrak{X}_{\hat\uptau}$ if there exists a Lie-B\"acklund morphism $\mathfrak{Y} \rightarrow \mathfrak{X}$, which is a surjective submersion. 
We say that $\mathfrak{Y}$ is a trivial extension of $\mathfrak{X}$ is $\mathfrak{Y} = \mathfrak{X} \times \mathbb{T}^n$, for some $n$.  
\end{definition}

One may define an extension using a system of the form:
\begin{equation}\label{eq::tau_ext}
\uptau := \hat\uptau + \sum_{i=1}^n f_i(x,u,w) \partial_{x_i} + \sum_{j=1}^m \sum_{k=0}^\infty u_j^{(k+1)} \partial_{u_j^{(k)}},
\end{equation}
where $w$ stands for a finite number of coordinate functions of $\mathfrak{X}$.

\begin{example}
    Any diffiety is an extension of the point diffiety.
\end{example}

\begin{example}\label{ex::jet_triv}
  The jet space is a trivial extension of the time diffiety.
\end{example}

\begin{example}\label{ex::time_varying}
    A time varying system~\eqref{eq::system} defines a diffiety that is is an extension of the time diffiety, a trivial extension and an extension of the jet space.    
\end{example}

\begin{definition}
Given two extensions $\pi_1: \mathfrak{Y}_1 \rightarrow \mathfrak{X}$ and $\pi_2: \mathfrak{Y}_2 \rightarrow \mathfrak{X}$, a morphism from $(\mathfrak{Y}_1,\pi_1)$ to $(\mathfrak{Y}_2,\pi_2)$ is a Lie-Bäcklund morphism $f: \mathfrak{Y}_1 \rightarrow \mathfrak{Y}_2$, such that $\pi_1 = \pi_2 \circ f$.    
\end{definition}

\begin{definition}
\label{def::flat_extension}
We say that an extension $\mathfrak{Y}$ of $\mathfrak{X}$ is flat at some point $p \in \mathfrak{Y}$ if there exists an open set $\mathfrak{U} \subset \mathfrak{Y}$, such that $\mathfrak{U}$ is Lie-B\"acklund isomorphic to an open subset of a trivial extension of $\mathfrak{X}$.   

A flat diffiety is a flat extension of the time diffiety.
\end{definition}

\subsection{Strong accessibility}
\label{subsec::accessibility}

Throughout the paper, we shall assume that the Jacobian matrix $M:=\left ( \frac{\partial f_i}{\partial u_k} \right)_{ik}$ has full rank $m$. This condition will be referred as \textit{the full control condition} in the sequel. This assumption can always be satisfied by incorporating command variables into the state vector\footnote{We can always restrict to this case by keeping $s:=\mathrm{rank}M$ controls, say $u_1, \ldots, u_s$ up to a permutation, completed with new controls $\dot u_{s+1}, \ldots, \dot u_m$, and taking $u_{s+1}, \ldots, u_m$ for new state variables.}.

Moreover the system is assumed to be \emph{strongly accessible}~\cite{Sussmann1972, Fliess1997}. We recall the definition. 

\begin{definition}\label{def::accessibility}
   A system~\eqref{eq::system} is \emph{strongly accessible} if
 for every $T > 0$, the set of \emph{reachable points} at time $T$
$$
\mathcal{R}(x_0,T) = \{x(T) |  x(0) = x_0, \forall t \in [0,T], \dot{x}(t) = f(x(t),u), u(.) \text{ analytic}\}
$$
has a non-empty interior for any initial condition $x_0$ in a dense open set.
\end{definition}
  
According to Fliess \textit{et al.}~\cite{Fliess1997}, the strong
accessibility is equivalent to the following property: for all
function $F: \R^n \rightarrow \R$ of $x$, if $F$ is constant along the
trajectory of the system, it is merely a constant, that is if $\uptau F
= 0$ then $F$ is constant on $\R^n$.
If the strong accessibility assumption is not satisfied, the system is constrained to evolve in a leaf of a given foliation. Then on this leaf, it is strongly accessible. The following proposition provides an easy criterion. Some more details will be given in subsection~\ref{subsec::not_stringly_access}.

\begin{proposition}\label{prop::accessibility}
Consider the system defined by the Cartan field $\tau$. 
Let $\updelta$ be a derivation. We recursively define $\hat\uptau^0\updelta:=\updelta$ and $\hat\uptau^{k+1}\updelta:=[\hat\uptau^{k}\updelta,\uptau]$.
The system is strongly accessible iff the Lie algebra generated by the $\hat\uptau^k\partial_{u_{j}}$, $1\le j\le m$, $k\in\mathbb{N}$ has dimension $n+m$. 
\end{proposition}
\begin{proof} Theorem 1.4~iii) of \cite{Fliess1997} implies that this is equivalent to the strong accessibility condition of Sussmann and Jurjevic~\cite[th.~3.1]{Sussmann1972}
\end{proof}

 \subsection{Linearized system}\label{subsec::linearized}

 To any diffiety $V$, we may associate its \emph{tangent extension}. If the diffiety is defined by a control system~\eqref{eq::system}, the tangent extension is defined by the original system together with the system
 \begin{equation}\label{eq::linearized}
     \mathrm{d}x_i=\mathrm{d}f_i,
 \end{equation}
where 
\begin{equation}\label{eq::d}
     \mathrm{d}f_i=\sum_{k=1}^{n}\partial_{x_k}f_i(x,u,t)\mathrm{d}x_k+
     \sum_{\ell=1}^{m}\partial_{u_{\ell}}f_i(x,u,t)\mathrm{d}u_l.
 \end{equation}

This new explicit system defines a diffiety extension and the $\mathrm{d}x_i$ as well as the $\mathrm{d}u_i^{(k)}$ may be just seen as convenient symbols for new coordinate functions. The prop.~\ref{prop::diff} shows that $\mathrm{d}$ defines a derivation that satisfies indeed the universal property of Kähler differentials in algebra.

\begin{definition}\label{def::tangent}
    The tangent extension $\mathrm{T}V$ of $(V,\uptau)$ (the diffiety associated to system~\eqref{eq::system}) is $V\times\mathbb{R}^{n}\times\left(\mathbb{R}^{\mathbb{N}}\right)^m$ equipped with the derivation $\tilde\uptau:=\uptau+\uptau_2$, where:
\begin{equation}\label{eq::hat_tau}
    \uptau_2:= \sum_{i=1}^n \mathrm{d}f_i\partial_{\mathrm{d}x_i}+
         \sum_{j=1}^m\sum_{k\in\mathbb{N}} \mathrm{d}u_j^{(k+1)}\partial_{\mathrm{d}u_j^{(k)}}.
\end{equation}
\end{definition}
It is easily checked that this is precisely the diffiety associated to the union of the systems \eqref{eq::system} and \eqref{eq::linearized}, according to def.~\ref{def::diff_sys}. For the sake of simplificity, we shall call this extended system, merely \textit{the linearized system}.

\begin{remark}
    Note that in this system the functions $\mathrm{d}u_j$  ${1\leq j \leq m}$ are new commands, in addition to the former commands $(u_j)_{1 \leq j \leq m}$.
\end{remark}

More generally, it is enough here just to consider the $\mathrm{d}x_i$ and $\mathrm{d}u_{j}^{(k)}$ as new coordinate functions, but the next proposition shows that $\mathrm{d}$ defines a universal derivation that satisfies the universal property that characterizes the module of differentials.

\begin{proposition}\label{prop::diff}
    Let $\updelta$ be any derivation from the ring $\mathcal{O}$ of functions to a differential $\mathcal{O}$ module $M$, then there exists a unique differential module morphism $\phi:\mathcal{O}(V)[\mathrm{d}x]\mapsto M$ such that $\updelta=\phi\circ\mathrm{d}$, where $\mathcal{O}[\mathrm{d}x]$ denoted the differential $\mathcal{O}(\mathrm{T}V)$ submodule generated by $\mathrm{d}x$.
\end{proposition}
\begin{proof} One must have $\phi\mathrm{d}x_i=\updelta x_i$, and this choice is enough to define $\phi$.
\end{proof}

We will now prove a useful result: the linearized system of a strongly accessible system defines a flat extension. For this, we will need a few easy lemmas.

\begin{lemma}\label{lem::transfer}
    With the notations of def.~\ref{def::tangent}, let
$$
\updelta=\sum_{i=1}^n a_{i} \partial_{x_i} + \sum_{j=1}^m \sum_{k=0}^\infty b_{j,k} \partial_{u_j^{(k)}},
$$
where the $a_i$ and $b_{j,k}$ belong to $\mathcal{O}(V)$,
and 
$$
\updelta_2=\sum_{i=1}^n a_{i} \partial_{\mathrm{d}x_i} + \sum_{j=1}^m \sum_{k=0}^\infty b_{j,k} \partial_{\mathrm{d}u_j^{(k)}}.
$$
Assuming that 
$$
\hat\uptau   \updelta = [\updelta,\tau]\footnote{According to proposition~\ref{prop::accessibility}} = \sum_{i=1}^n \hat a_{i} \partial_{x_i} + \sum_{j=1}^m \sum_{k=0}^\infty \hat b_{j,k} \partial_{u_j^{(k)}},
$$
then 
$$
\hat\uptau_2\updelta_2= [\updelta_2,\uptau_2] = \sum_{i=1}^n \hat a_{i} \partial_{\mathrm{d}x_i} + \sum_{j=1}^m \sum_{k=0}^\infty \hat b_{j,k} \partial_{\mathrm{d}u_j^{(k)}},
$$
using the notation $\uptau_2$ of \eqref{eq::hat_tau}.
\end{lemma}
\begin{proof}
Easy computations show that we have indeed 
$$\begin{array}{ll}
\hat\uptau\updelta=[\updelta,\uptau] &=
\sum_{i=1}^n \left(-\dot{a}_i+\sum_{\ell=1}^{n}  a_k\partial_{x_k}f_i + \sum_{\lambda=1}^{m} b_{j,0}\partial_{u_j}f_i\right)\partial_{x_i}\\
&\hphantom{=}+\sum_{j=1}^m \sum_{k=0}^\infty \left(b_{j,k+1}\partial_{u_{j,k+1}}u_{j,k+1}-\dot b_{j,k}\right)\partial_{u_k^{(k)}}\\
     &=
\sum_{i=1}^n \left(-\dot{a}_i+\sum_{\ell=1}^{n}  a_k\partial_{\mathrm{d}x_k}\mathrm{d}f_i + \sum_{\lambda=1}^{m} b_{j,0}\partial_{\mathrm{d}u_j}\mathrm{d}f_i\right)\partial_{x_i}\\
&\hphantom{=}+\sum_{j=1}^m \sum_{k=0}^\infty \left(b_{j,k+1}\partial_{\mathrm{d}u_{j,k+1}}\mathrm{d}u_{j,k+1}-\dot b_{j,k}\right)\partial_{u_k^{(k)}},
\end{array}
$$
so that
$\hat b_{j,k}=b_{j,k+1}  -\dot{b}_{j,k}$, which is also the corresponding coefficient of $\hat\uptau_2\updelta_2$. Also, we have
    $$
\hat a_i= -\dot{a}_i+\left(\sum_k a_k\partial_{x_k}f_i + \sum_j b_{j,0}\partial_{u_j}f_i\right)=-\dot{a}_i+\left(\sum_k a_k\partial_{\mathrm{d}x_k}\mathrm{d}f_i + \sum_j b_{j,0}\partial_{\mathrm{d}u_j}\mathrm{d}f_i\right),
    $$
which is also the expression of the corresponding coefficient of $\hat\uptau_2\updelta_2$.
\end{proof}

\begin{lemma}\label{lem::involutive}
    We will extend the notation $\hat\uptau^{k} \partial_{u_j}$ by setting $\hat\uptau^{-k} \partial_{u_j}=\partial_{u_j^{(k)}}$.
Let $\Gamma:=\langle\hat\uptau^k\partial_{u_j}| 1\le j\le m, k\in\mathbb{Z}\rangle$, then
 $\Gamma$ is involutive.
\end{lemma}
\begin{proof}
    As for $k \geq 0$, $\hat\uptau^k\partial_{u_j}\in\langle\partial_{x_i}|1\le i\le n\rangle+\langle\partial_{u_j}|1\le j\le m\rangle$, we have for any $k\in\mathbb{N}$ and any $1\le j\le m$ $\hat\uptau^k\partial_{u_j}\in\langle\hat\uptau^{\ell}\partial_{u_{\hat{\mathrm{Im}ath}}}|0\le\ell\le n,\>1\le \hat{\mathrm{Im}ath}\le m\rangle$, so that \begin{equation}\label{eq::Gamma}
    \Gamma=\langle\hat\uptau^k\partial_{u_j}| 1\le j\le m, -\infty\le k\le n\rangle.
    \end{equation}

    We prove by recurrence on $\ell$ that for any $1\le h\le m$, any $1\le j\le m$ and any $-\infty\le k\le n$,$[\hat\uptau^\ell\partial_{u_h}, \hat\uptau^k\partial_{u_j}]\in\Gamma$. No derivative of the controls of order greater than $n-1$ appears in $\hat\uptau^k\partial_{u_j}$ for $1\le j\le m$ and $-\infty\le k\le n$, so that indeed $[\hat\uptau^{\ell}\partial_{u_\ell},\hat\uptau^k\partial_{u_j}]=0\in\Gamma$ when $\ell\le-n$. 
    
    Assume the result stands for $\ell\le s$, then by Jacobi identity:
$$   
[\hat\uptau^{s+1}\partial_{u_j},\hat\uptau^k\partial_{u_{j_2}}]=
[[\hat\uptau^{s}\partial_{u_j},\uptau],\hat\uptau^k\partial_{u_{j_2}}]=
[\hat\uptau^{s}\partial_{u_j},\hat\uptau^{k+1}\partial_{u_{j_2}}]
-[[\hat\uptau^{s}\partial_{u_j},\hat\uptau^k\partial_{u_{j_2}}],\uptau]\in\Gamma,
$$
so that it stands also for $s+1$. \ck
\end{proof}

\begin{theorem}\label{th::lin_access}
    The system \eqref{eq::system} is strongly accessible iff the extension defined by \eqref{eq::linearized} is flat.
\end{theorem}
\begin{proof} By prop.~\ref{prop::accessibility}, the system is strongly accessible iff the Lie algebra generated by the $\hat\uptau^k\partial_{u_{j}}$, $1\le j\le m$, $k\in\mathbb{N}$ has dimension $n+m$. The generators of this Lie algebra are included in $\Gamma$, which is involutive by lemma~\ref{lem::involutive}, so that this Lie algebra is in fact included in $\Gamma$.

 Let $\Gamma_2:=\langle\hat\uptau_2^k\partial_{\mathrm{d}u_j}| 1\le j\le m, k\in\mathbb{Z}\rangle$. By lemma~\ref{lem::transfer}, the strong accessibility of the system~\eqref{eq::system} is equivalent to the fact that the Lie algebra generated by  $\hat\uptau_2^k\partial_{\mathbb{d}u_{j}}$, $1\le j\le m$, $k\in\mathbb{N}$ has dimension $n+m$, which is equivalent to the fact that the extension defined by \eqref{eq::linearized} is strongly accessible. 

 This is equivalent to the fact that the module defined by \eqref{eq::linearized} has no torsion element, which is equivalent to it being a free module~\cite{Fliess1990}, meaning that \eqref{eq::linearized} defines a flat diffiety extension.
\end{proof}

\section{One parameter local automorphisms}\label{sec::one_parameter}

The aim of this work is to study a special class of pseudogroups of automorphisms of a diffiety, \textit{viz.} the one parameter pseudogroups generated by special types of vector field.  In this section, we will study the basic properties of such fields.

In the sequel, for simplicity, we will restrict to systems that do not depend on the time explicitly. Necessary adaptations are easy but involve more lengthy computations.

Then the Cartan field is given by:
$$
\uptau = \sum_{i=j}^n f_j \partial_{x_j} + \sum_{k=1}^m \sum_{\ell=0}^\infty u_k^{(\ell+1)} \partial_{u_k^{(\ell)}}.
$$

\subsection{Infinitesimal symmetries of the diffiety}
\label{subsec::infinitesimal}

We are looking for \emph{infinitesimal symmetries} of a diffiety, \textit{i.e.} the vector fields that commute with the Cartan field. For such a vector field, the flow if it exists, defines a local group of Lie--B\"acklund automorphisms. In the sequel, such a vector field can also be called \emph{a Lie--B\"acklund compatible vector field}. Geometrically this means that the flow of such a vector field induces a local automorphism of the diffiety that leaves invariant the set of local integral curves of the Cartan field.  

\begin{lemma}
\label{lem::commutativity}
A vector field $\updelta = \sum_{j=1}^n a_j \partial_{x_j} + \sum_{k=1}^n \sum_{\ell=0}^\infty b_{k,\ll} \partial_{u_k^{(\ell)}}$ is an infinitesimal symmetry, \textit{i.e.} commutes with $\uptau$, if and only if:
\begin{eqnarray}
\label{eq::commutativity1}
\updelta f_j & = & \uptau a_j, \forall j \in \{1, \cdots, n\}, \\
\label{eq::commutativity2}
b_{k,\ell+1} & = & \uptau b_{k,\ell}, \forall k \in \{1, \cdots, m\}, \forall \ell \in \N.
\end{eqnarray}
\end{lemma}
\begin{proof} The Lie Bracket with $\uptau$ reads to:
$$
[\updelta,\uptau] = \left ( \sum_{j=1}^n (\updelta f_j) \partial_{x_j} + \sum_{k=1}^m \sum_{\ell=0}^\infty (\updelta u_k^{(\ell+1)}) \partial_{u_k^{(\ell)}} \right ) - 
\left ( \sum_{j=1}^n (\uptau a_j) \partial_{x_j} + \sum_{k=1}^m \sum_{\ell=0}^\infty (\uptau b_{k,\ell}) \partial_{u_k^{(\ell)}}   \right ). 
$$

Indeed the further terms cancel each other. This expression vanishes if, and only if, we have: $\updelta f_j = \uptau a_j$ for all $j$ and $\updelta u_k^{(\ell+1)} = b_{k,\ell+1} = \uptau b_{k,\ell}$.
\end{proof}

Therefore finding such a vector field is equivalent to solving equations~\eqref{eq::commutativity1} and ~\eqref{eq::commutativity2}: 
$$
\updelta f_j =  \uptau a_j, \forall j \in \{1, \cdots, n\},
$$
that define a diffiety extension equivalent to the linearized system~\eqref{eq::linearized}.
The functions $b_{k,0}$ are arbitrary and by derivation determine the functions $b_{k,\ell}$ for $\ell > 0$ inductively through  equations~\eqref{eq::commutativity2}.   

Eventually, this system has $n$ equations and $n+m$ unknowns which are the functions $a_j,b_{k,0}$ with $1 \leq j \leq n$ and $1 \leq k \leq m$. The fact the system is underdetermined allows to find infinite spaces of solutions in general. See section~\ref{sec::parametrization} for details. 

\begin{remark}\label{rem::controls}
For simplicity, we may reduce to the case $u_i=\dot{x}_{i}$ for $1 \leq i \leq m$, up to a permutation of indices,
and
consider a system 
    \begin{equation}
        \label{eq::sysbis}
        x_i=f_i(x,\dot{x}_{1}, \ldots, \dot{x}_{m},t), \text{ for } m<i\le n,
    \end{equation}
    defining a diffiety $V$. With this notation, we have $b_{i,k} = a_i^{(k+1)}$ and any infinitesimal symmetry $\updelta$ can be written $\updelta=\sum_{i=1}^{m} \sum_{k\in\mathbb{N}}a_i^{(k)}\partial_{x_i^{(k)}}+\sum_{k=m+1}^{n} a_i\partial_{x_i}$.
\end{remark}
We will make this assumption in the sequel, as it allows to reduce to solving only equations~\eqref{eq::commutativity1}.

\subsection{Integrable infinitesimal symmetries and flow}\label{subsec::int_sym}

We need to consider infinitesimal symmetries with a special property, ``integrability'', that will allow us in the next section to associate to them a one parameter pseudogroup of diffiety morphisms.

\begin{definition}
\label{def::integrability}
An infinitesimal symmetry $\updelta$ on the diffiety $\mathfrak{X}$ is said to be integrable if $e := \max\{\max_j \max_\ell \ord_\uptau \, \updelta^\ell x_j, \max_k \max_\ell \ord_\uptau \, \updelta^\ell u_k \}< \infty$. Here the order $\ord_\uptau$ of a variable is defined as the highest order of derivative with respect to $\uptau$ that appears in its expression. 

The \emph{order of} $\updelta$ is $\ord\updelta:=\max_{i=1}^{n}\ord_{\uptau}a_i=\updelta x_i$ and we call $e$ the \emph{iterated order} of $\updelta$.
\end{definition}

We will show that this integrability condition implies the existence of a one parameter pseudogroup defined on a dense open subset $\mathfrak{U} \subset \mathfrak{X}$ of the diffiety.
The existence of such a pseudogroup implies in turn the existence of an integral curve passing through all points in $\mathfrak{U}$.

\begin{remark}
   Note that using smooth functions, the existence of integral curves is only a necessary condition of integrability and not a sufficient one. Indeed, Peano~\cite[p.~\textsc{xviii}]{Genocchi1884} and Borel~\cite{Borel1995} have shown the existence of a smooth function that admits any given Taylor series at the origin, so that the non integrable field $\sum_{k\in\mathbb{N}} u^{(k+1)}\partial_{u^{(k)}}$ admits a smooth integral curve. (See also Hörmander~\cite[p.~16]{Hormander1990}). In the analytic case, the existence of an analytic integral curve obviously depends on the convergence of the series. 
\end{remark}

We need first a preliminary lemma, showing that we can associate to integrable symmetries spaces of finite dimension.

\begin{lemma}
\label{lem::integrability}
Let the vector field $\updelta$ be integrable in the sense of definition~\ref{def::integrability}. 

i) There exists a dense open set $\mathfrak{U} \subset \mathfrak{X}$ and natural numbers $r_1, \ldots, r_n$ such that in the neighborhood of any point $\eta\in\mathfrak{U}$, there exist analytic functions $\phi_{\eta,i}$, for $1\le i\le n$, such that:

{\leftskip=1cm
    
    \vphantom{a}\llap{a)} the set of variables $Y:=\{\updelta^k x_i|1\le i\le n,\> 0\le k<r_i\}$ is functionally independent;

    \vphantom{a}\llap{b)} and we have
\begin{equation}
\label{eq::integrability}
\updelta^{r_i} x_i = \phi_{\eta,i}(Y), \mbox{ for } 1 \leq i \leq n.
\end{equation}

}

ii) For any function $F\in\mathcal{O}(\mathfrak{X})$, there exist a dense open set $U_F$ and $r_F\in\mathbb{N}$, such that in the neighborhood of any point $\eta\in U_F$, there exists an analytic functions $\phi_{F,\eta}$,  such that:

{\leftskip=1cm
    
    \vphantom{a}\llap{a)} the set of variables $Z:=\{z_{k}:=\updelta^k F| 0\le k<r_F\}$, is functionally independent;

    \vphantom{a}\llap{b)} and we have
\begin{equation}
\label{eq::integrability2}
\updelta^{r_F} F = \phi_{F,\eta}(Z).
\end{equation}
}
\end{lemma}
\begin{proof}
i) The integrability condition~\ref{def::integrability} implies that the maximal order $e$ of derivatives appearing in the $\updelta^k x_i$ is finite. These derivatives belong to the set $\Upsilon_e:=\{x_i^{(k)}| 1\le i\le m,\> 0\le k\le e\}\cup \{x_i|m<i\le n\}$. The generic rank of the Jacobian matrix 
$$
\left(\>\>\frac{\partial \updelta^k x_i}{\partial\upsilon}\>\>\Bigg|
\begin{array}{l}\scriptstyle 1\le i\le n\\ 
\scriptstyle 0\le k\le n+m(e+1)+1\\
\scriptstyle \upsilon\in\Upsilon_e\end{array}\right)
$$ is at most $\rho\le n+m(e+1)$. So, there exists sets of derivatives $Y=\{\updelta^k x_i|1\le i\le n,\> 0\le k<r_i\}$, with $0<r_i\le n+m(e+1)+1$ and $Y_j=\{\updelta^k x_i|1\le i\le n,\> 0\le k<r_i\>\hbox{for}\> i\neq j,\> k\le r_i\>\hbox{for}\> i=j\}$, such that
$$
\rank \left(\>\>\frac{\partial y}{\partial\upsilon}\>\>\Bigg|
\begin{array}{l}y\in Y_i\\ 
\upsilon\in\Upsilon_e\end{array}\right)=
\rank \left(\>\>\frac{\partial y}{\partial\upsilon}\>\>\Bigg|
\begin{array}{l}y\in Y\\ 
\upsilon\in\Upsilon_e\end{array}\right),\>\hbox{for all}\> 
1\le i\le n.
$$
One may choose such a set $Y$ that is minimal for inclusion. Let then $\mathfrak{U}$ be the open set of points $\eta$ where the rank of this matrix is equal to $e$. 

For any point $\eta\in\mathfrak{U}$, there exists a set of derivatives $\hat\Upsilon\subset\Upsilon_e$, with $\sharp \hat\Upsilon=\sharp Y=\rho$, such that 
$$
\left|\>\>
\frac{\partial y}{\partial\upsilon}\>\>\bigg|
\begin{array}{l}y\in Y\\ 
\upsilon\in\hat\Upsilon\end{array}
\right|
$$ does not vanish at $\eta$. By the implicit function theorem~\cite[p.~138]{Cartan1961}, one may find functions $F_\upsilon$, such that $\upsilon=F_\upsilon(y\in Y, \tilde\upsilon\in\Upsilon_e\setminus\hat\Upsilon)$ for all $\upsilon\in\hat\Upsilon$. Then $\updelta^{r_i}x_i$ is a function $G_i(\upsilon\in\Upsilon_e)$  by the integrability hypothesis. We may write $\updelta^{r_i}x_i=G_i(F_\upsilon(y\in Y;\upsilon\in\Upsilon_e\setminus\hat\Upsilon))|\upsilon\in\hat\Upsilon; \upsilon\in\Upsilon_e\setminus\hat\Upsilon)$. Assume this last expression actually depends on some $\upsilon_0\in\Upsilon_e\setminus\hat\Upsilon$, then the rank of 
$$
\left(\>\>\frac{\partial y}{\partial\upsilon}\>\>\Bigg|
\begin{array}{l}y\in Y_i\\ 
\upsilon\in\Upsilon_e\end{array}\right)
$$
would be strictly greater than $\rho$, a contradiction. So that 
$\updelta^{r_i}x_i=\phi_{\eta,i}(y\in Y):=G_i(F_\upsilon(y\in Y);\upsilon\in\Upsilon_e\setminus\hat\Upsilon)|\upsilon\in\hat\Upsilon; \upsilon\in\Upsilon_e\setminus\hat\Upsilon)$, which completes the proof.

ii) The proof is exactly similar.
\end{proof}

By lem.~\ref{lem::integrability}, one may consider the action of $\updelta$ on the projection $\mathfrak{U}_0$ of the open set $\mathfrak{U}$ on the finite space corresponding to the coordinate functions $Y$. In the same way, for any function $F$ on the diffiety, we can also consider the finite space defined by its successive $\updelta$ derivatives. On this space of finite dimension, one is able to use classical results on the flow of a vector field (see \cite[p.~61]{Krantz2002}, \cite[p.~268]{Kobayashi1963} or \cite[p.~50]{Narasimhan1985} for more details).

\begin{remark}\label{rem::g}
    In the sequel $g_F$ denotes a function $\mathbb{R}\mapsto \mathcal{O}(\mathfrak{U})$, so that $g_F(s)(\eta)$ for $(s,\eta)\in\mathbb{R}\times\mathfrak{U}$ makes sense. We will rather use the notation $g_F(s,\eta)$ in such a case for better clarity.
\end{remark}

\begin{lemma}\label{lem::gamma}
    i) For any point $\hat y\in\mathfrak{Y}_0$ there exists a unique solution $\gamma_{\hat y}:]\epsilon_{-}(\hat y),\epsilon_{+}(\hat y)[\mapsto\mathfrak{Y}_0$ of the system
\begin{equation}
\label{eq::sys_gamma}
\begin{array}{rll}
\frac{\partial \gamma_{\hat y,i,k}}{\partial s}(s)&= \gamma_{\hat y,i,k+1}(s)&\hbox{if}\> k<r_i-1 \\
     & = \phi_{\hat y,i}(\gamma_{\hat y}(s))&\hbox{if}\> k=r_i-1, 
\end{array}
\end{equation}
with initial conditions
\begin{equation}
\label{eq::initial_cond}
\gamma_{\hat y,i,k} = \updelta^k x_i(\hat y),\>\hbox{for}\> 1\le i\le n \>\hbox{and}\> 0\le r< r_i,   
\end{equation}
and the interval $]\epsilon_{-}(\hat y),\epsilon_{+}(\hat y)[$ is maximal with this property.

ii) Furthermore, the function $(\hat y,s)\mapsto\gamma_{\hat y,i,k}(s)$ is an analytic function
from an open set of $\mathfrak{Y}_0 \times \R$. 

iii) For any function $F\in\mathcal{O}(\mathfrak{X})$, there exists an open subset $U_F\subset\mathfrak{X}$, such that for any point $\eta\in U_F$, there exists a locally unique analytic function $g_{F}: U_F\times\mathbb{R}\mapsto \mathcal{O}(\mathfrak{X})$, $s\mapsto g_{F}(s)$, defined in a non empty open neighborhood $U_{F,\eta}$ of $\eta\in U_F$ for all $s\in I(F,\eta)$, with $0\in I(F,\eta)\subset\mathbb{R}$ and which is solution of 
\begin{equation}\label{eq::F_delta}
\begin{array}{ll}
    \frac{\partial g_{F}}{\partial s}&=\updelta g_{F}(s)\\
    g_F(0)&=F.
\end{array}
\end{equation}

iv) If $s_1+s_2\in I$, then $g_{F}(s_1+s_2)=g_{g_F(s_1)}(s_2)$.

v) We have 
\begin{equation}\label{eq::F_tau}
    \uptau g_{F}(s)= g_{\uptau F}(s).
\end{equation}
vi) We may identify $\eta\mapsto \gamma_{\pi\eta,i,0}$ with $g_{x_i}(\cdot,\eta)$, defined on $\mathfrak{U}$, and $g_{x_i^{(k)}}=\uptau^k g_{x_i}$ is such that $g_{x_i^{(k)}}(\cdot,\eta)$ can also be defined for all $s\in ]\epsilon_{-}(\pi(\eta)),\epsilon_{+}(\pi(\eta))[$. With these notations, if $F$ depends on the derivatives $\upsilon\in\Upsilon_0$, with $\Upsilon_0$ finite, 
\begin{equation}\label{eq::composition1}
    g_{F}(\cdot,\cdot)=F(g_{\upsilon}(\cdot,\cdot)|\upsilon\in\Upsilon_0)
\end{equation}
and
\begin{equation}\label{eq::composition2}
    g_{F}(s_1+s_2,\cdot)=g_{F}(s_1,g_{\upsilon}(s_2,\cdot)|\upsilon\in\Upsilon_0).
\end{equation}
\end{lemma}
\begin{proof}
    i) The existence of a $\cal{C}^\infty$ solution on a maximal interval is a consequence of the elementary theory of differential equations. 
    
    ii) The fact that this solution is analytic is a consequence of Cauchy -- Kovalevskaya theorem~\cite{Kovalevskaya1875}. We refer to~\cite[p.~50, corollary 1.8.11]{Narasimhan1985} for more details.

    For iii), we proceed in the same way, using lem.~\ref{lem::integrability}~ii). Indeed, using the notations of the lemma, we only have to solve the system
\begin{equation}
\label{eq::sys_g}
\begin{array}{rll}
\frac{\partial g_k}{\partial s}(s,\cdot)&= g_{k+1}(s,\cdot)&\hbox{if}\> k<r_F-1 \\
     & = \phi(g_{z_0}(s,\cdot), \ldots, z_{r_F}(s,\cdot))&\hbox{if}\> k=r_F-1, 
\end{array}
\end{equation}
with initial conditions
\begin{equation}
\label{eq::initial_cond_g}
g_{z_k}(0,\cdot) = \updelta^kF(\cdot),\>\hbox{for}\> 0\le k<r_F-1.   
\end{equation}
As for ii), such a system admits a locally unique analytic solution, in $s$ and the $z_k$, by the Cauchy -- Kovalevskaya theorem.

iv) Consider $\hat g(s,\cdot):=g_F(s_1+s,\cdot)$. We have $\hat g(0,\cdot)=g_F(s_1,\cdot)$ and $\partial_s \hat g(s,\cdot)=\updelta \hat g(s,\cdot)$, a system that locally defines $s\mapsto g_{g_F(s_1+s,\cdot)}$, so that $g_{g_F(s_1)}(s_2,\cdot)=\hat g(s_2,\cdot):=g_F(s_1+s_2,\cdot)$. 

v) By iii), we have
\begin{equation}\label{eq::p_delta}
    \partial_s (\uptau g_F(s)-g_{\updelta F}(s)) =\updelta (\uptau g_F(s)-g_{\updelta F}(s))
\end{equation}
and $\uptau g_F(0)-g_{\updelta F}(0)=0$. The only local analytic solution of eq.~\ref{eq::p_delta} with initial condition $0$ is $\uptau g_F(s)-g_{\updelta F}(s)=0$, hence the result.

vi) The identification is clear by iii) as $\gamma_{\pi\eta,i,0}$ is solution of $\partial_s \gamma_{\pi\eta,i,0}=\updelta \gamma_{\pi\eta,i,0}$ with $\gamma_{\pi\eta,i,0}(0)=x_i$.
The functions $g_{x_i^{(k)}}$ as defined in lem.~\ref{lem::gamma}~vi) can be extended, fixing $\eta$, by analytic prolongation with respect to $s$, so that they are all defined on the same maximal interval $]\epsilon_{-}(\pi\eta),\epsilon_{+}(\pi\eta)[$ as the $\gamma_{\pi\eta,i,k}$.

In the same way, we have
$$
\frac{\partial F(g_{\upsilon}(s))}{s}
=\sum_{\upsilon\in\Upsilon_0} 
  \frac{\partial F}{\partial\upsilon} 
  \frac{\partial g_{\upsilon}(s)}{\partial s}
= \sum_{\upsilon\in\Upsilon_0} 
  \frac{\partial F}{\partial\upsilon} 
  \updelta g_{\upsilon}(s)
= \updelta F(g_{\upsilon}(s))
$$
and $F(g_{\upsilon}(0))=F$, so that \eqref{eq::composition1} is a consequence of iii). Then \eqref{eq::composition2} is a consequence of \eqref{eq::composition1} and iv).
\end{proof}

We are then able to build the flow mappings that will belong to the pseudogroup to be defined in the next section, with the following definition.

\begin{definition}\label{def::flow_mapping}
    Let $U\subset\mathfrak{U}$ be an open subset and $U_0\subset\mathfrak{U}_0=\pi(U)$ be its projection on the space of coordinates $Y$. We assume first that $U=\pi^{-1}(U_0)$.

    Let $I(U_0)=\bigcap_{y\in U_0}]\epsilon_{-}(y),\epsilon_{+}(y)[$. 
    For any $s\in I(U_0)$, we can define $\mathfrak{g}_U(s):U_0\mapsto\mathfrak{U}$ by setting in the coordinates $x_i^{(k)}$:
\begin{equation}\label{eq::flow}
\begin{array}{lll}
 \mathfrak{g}_U(s):&U&\mapsto \mathfrak{U}  \\
     & (x_i^{(k)}| {\scriptstyle 1\le i\le m;\> k\in\mathbb{N}\>\hbox{and}\>
                   m<i\le n;\> k=0})
     &\mapsto
     (g_{x_i^{(k)}}(s)| {\scriptstyle 1\le i\le m;\> k\in\mathbb{N}\>\hbox{and}\>
                   m<i\le n;\> k=0}).
\end{array}   
\end{equation}

If $U\neq\pi^{-1}\circ\pi(U)$, we may find thanks to rem.~\ref{rem::topology} $r\in\mathbb{N}$ so that $U=\pi_r^{-1}\circ\pi_r (U)$. We only need to consider then the greatest $\epsilon_{+}(\eta)$ and the smallest $\epsilon_{-}(\eta)$ such that the $g_{x_i^{(k)}}(s)$ are defined for all $k\le r$ and all $s\in]\epsilon_{-}(\eta),\epsilon_{+}(\eta)[$. It is then enought to define $I(U):=\bigcap_{\eta\in U} ]\epsilon_{-}(\eta),\epsilon_{+}(\eta)[$.
\end{definition}
We notice that, for any non empty open subset $u\subset\mathfrak{U}$, $0\in I(U)$ and $\mathfrak{g}_U(0)=\mathrm{Id}_{|U}$.

Let us show first that this defines diffiety morphisms.

\begin{theorem}\label{th::diff_morphism}
The mapping $\mathfrak{g}_{U}$ is a diffiety morphism.
\end{theorem}
\begin{proof} By lem.~\ref{lem::gamma} v), assuming that $F$ depends on a finite set of derivatives $\Upsilon_0$, we have for any $F\in\mathcal{0}(\mathfrak{U})$:
$$
\uptau \mathfrak{g}_{U}(s)^{\ast}F
=\sum_{\upsilon\in\Upsilon_0} 
  \frac{\partial F}{\partial\upsilon} 
  \uptau g_{\upsilon}(s)
= \sum_{\upsilon\in\Upsilon_0} 
  \frac{\partial F}{\partial\upsilon} 
  g_{\upsilon'}(s)
= \mathfrak{g}_U(s)^{\ast}(\updelta F),
$$
so that $\mathfrak{g}_U$ is compatible with $\uptau$ as in def.~\ref{def::diffiety}.
\end{proof}

\subsection{The pseudogroup defined by an integrable infinitesimal symmetry}\label{subsec::pseudogroup}

We now need now prove that this set of morphisms forms a one parameter pseudogroup, according to the following definition. See e.g.\ Kobayashi and Nomizu~\cite[p.~1--2]{Kobayashi1963} or Evtushik \cite[def.~5]{Evtushik1999} for a general definition.

\begin{definition}\label{def::pseudogroup}
  i) A pseudogroup acting on a diffiety $\mathfrak{V}$ is a set $\mathfrak{G}$ of diffiety isomorphisms between open subsets of $\mathfrak{V}$, such that:
  
{\leftskip=1cm
    
    \vphantom{a}\llap{a)} \emph{Cover}. --- For any point $\eta\in \mathfrak{V}$ there exists $\mathfrak{g}\in\mathfrak{G}$ such that $\mathfrak{g}$ is defined on $W\ni\eta$;
    
    \vphantom{a}\llap{b)} \emph{Restriction}. --- The restriction of an element g in G to any open set contained in its domain is also in G;
    
    \vphantom{a}\llap{c)} \emph{Composition}. --- The composition of two elements $\mathfrak{g_1},\mathfrak{g_2}\in\mathfrak{G}$, such that $\mathrm{Im}\mathfrak{g_1}\cap\mathrm{Dom}\mathfrak{g_2}\neq\emptyset$, is in $\mathfrak{G}$;

    \vphantom{a}\llap{d)} \emph{Neutral element}. --- For every open set $U$, $\mathrm{Id}_{|U}\in\mathfrak{G}$;
    
    \vphantom{a}\llap{e)} \emph{Inverse}. --- The inverse of an element $\mathfrak{g}\in\mathfrak{G}$ is in $\mathfrak{G}$;
    
    \vphantom{a}\llap{f)} \emph{Local}. --- If $\mathfrak{g}: U\mapsto V$ is a homeomorphism between open sets of $S$ and $U=\bigcup_{\ell\in L} U_\ell$, with $\mathfrak{g}\in\mathfrak{G}$, for all $\ell\in L$, then $\mathfrak{g}\in\mathfrak{G}$, \emph{provided that $U$ is connected}.

}

  ii) A one parameter pseudogroup acting on a $\mathfrak{V}$ is a pseudogroup $\mathfrak{G}$ such that there exists an application $\mathrm{Dom}$ from the set of open subsets of $\mathfrak{V}$ to the set of connected subsets of $\mathbb{R}$ containing $0$, such that: 
  
 {\leftskip=1cm
 
      \vphantom{a}\llap{a)} \emph{Existence}. --- For any open subset $U\subset\mathfrak{V}$, for any $s\in\mathrm{Dom}(U)$, there exists a mapping $\mathfrak{g}_U(s)\in\mathfrak{G}$, defined on $U$ and for any point $\eta\in\mathfrak{V}$, there exists $U\ni\eta$ such that $\mathrm{Dom}(U)$ is open and $\mathfrak{g}_U$ is locally injective;
      
      \vphantom{a}\llap{b)} \emph{Unicity}. --- For any $\mathfrak{g}\in\mathfrak{G}$ defined on $U$ and different from identity, we have $\mathrm{Dom}(U)\neq\{0\}$ and there exists $s\in\mathrm{Dom}(U)$ such that $\mathfrak{g}=\mathfrak{g}_U(s)$;
      
      \vphantom{a}\llap{c)} \emph{Inverse}. --- If $\mathrm{Im}\mathfrak{g}_U(s)=V$, then if $\mathrm{Dom}(U)=-\mathrm{Dom}(V)$, and $\forall s\in\mathrm{Dom}(U)$, we have $\mathfrak{g}_V(-s)\circ \mathfrak{g}_U(s)=\mathrm{Id}_{|U}$;
      
      \vphantom{a}\llap{d)} \emph{Composition}. --- If $\mathrm{Im}\mathfrak{g}_U(s_1)=V$, then if $s_1\in\mathrm{Dom}(U)$ and $s_1+s_2\in\mathrm{Dom}(U)$, we have $s_2\in\mathrm{Dom}(V)$ and $\mathfrak{g}_V(s_2)\circ\mathfrak{g}_U(s_1)=\mathfrak{g}_U(s_1+s_2)$.

}
\end{definition}

``Existence'' basically means that we can associate to an open subset and an real number an element of the pseudogroup, ``Unicity'' that this is the only way to define pseudogroup elements, ``Inverse'' that the element defined on $V$ corresponding to $-s$ is the inverse of the element with image $V$ corresponding to $s$ and composition the the element corresponding to $s_1+s_2$ is the composition of the element defined by $s_1$ and the one defined by $s_2$, in an obvious way. These two last properties imply that $\mathfrak{G}_U(0)=\mathrm{Id}$.

\begin{remark}\label{rem::connected}
We have added the restriction \emph{provided that $U$ is connected} in i)~f), which does not appear in~\cite{Kobayashi1963}  for technical reasons. If not, we would have needed to admit mappings that are equal to  $\mathfrak{g}_{U_1}(s_1)$ on $U_1$ and $\mathfrak{g}_{U_2}(s_2)$ on $U_2$, with two different real values $s_1$ and $s_2$, which does not seem compatible with one could expect from a ``one parameter'' pseudogroup.
\end{remark}

We will now associate a one parameter pseudogroup structure to any integrable local symmetry.

\begin{theorem}
\label{thm::pseudo_group_flot}
Using definitions in def.~\ref{def::flow_mapping}, if $\updelta\neq0$, the set 
$$
\mathfrak{G}_{\updelta}:=\bigcup_{U\subset\mathfrak{U}}\{\mathfrak{g}_U(s)|s\in I(U)\},
$$
where $U$ denotes open subsets, defines a one parameter pseudogroup, using the mapping $I:U\mapsto U(U)$, as in def.~\ref{def::pseudogroup}.
\end{theorem}
\begin{proof}
To prove ii) a) \textit{Existence}, we first remark that $0\in\mathrm{Dom}(U)$ for any open set $U$ and that $\mathrm{Id}=\mathfrak{g}_U(0)$ is defined on $U$. According to rem.~\ref{rem::topology}, there exists $r\in\mathbb{N}$ such that $\pi_r^{-1}\pi_r\mathfrak{U}=\mathfrak{U}$ and it may be chosen so that $r\ge e$, where $e$ is the maximal order of the $\updelta^k x_i$. With this definition, the interval $]\epsilon_{-}(\eta),\epsilon_{+}(\eta)[$ only depends on $\pi_r\eta$, according to def.~\ref{def::flow_mapping}. 

Any point $\eta\in\mathfrak{U}$ belongs to an open subset $U$ such that the adherence $\overline{\pi_r U}$ in $\pi_r\mathfrak{U}$ is compact. Assume that $\overline{I(U)}=[\hat\epsilon_{-},\hat\epsilon{+}]$ and that $\hat\epsilon_{+}\in I(U)$. There exists a sequence of points $\eta_\ell$, $\ell\in\mathbb{N}$, such that $\epsilon_{+}(\eta_{\ell})$ tends to $\hat\epsilon_{+}$. We can extract a subsequence $\eta_{\ell_\lambda}$ such that $\pi_r(\eta_{\ell_\lambda})$ is convergent in the compact set $\overline{\pi_r U}$ and find $\eta_0\in U$ with $\pi_r(\eta_0)=\lim_{\lambda}\pi_r(\eta_{\ell_\lambda})$. Then, $\epsilon_{+}(\eta_0)=\hat\epsilon_{+}$ and $\hat\epsilon_{+}\notin I(U)$, a contradiction. We may proceed in the same way for $\hat\epsilon_{-}$.

As $\updelta\neq0$, there exist $\eta\in U$, such that $\updelta(\eta)\neq0$, which implies local injectivity. This concludes the proof of ii) a), of which i) a) is an easy consequence.

i) b), i) d) and ii) b) are straightforward consequences of def.~\ref{def::flow_mapping}.

ii) c) and d) are consequences of lem.~\ref{lem::gamma}~iv), that respectively imply i) c) and i) e).

We now need to prove i) f). For any diffiety homeomorphism $\hat{\mathfrak{g}}:U\mapsto V$ such that $U=\bigcup_{\ell\in L}U_\ell$ and $\hat{\mathfrak{g}}_{| U_\ell}=\mathfrak{g}_{U_\ell}(s_\ell)$, assume the result does not stand. Then, there exists an open set $U_0\subset U$, maximal for inclusion, such that $\hat{\mathfrak{g}}_{|U_0}=\mathfrak{g}_{U_0}(s)$ for some $s\in\mathrm{Dom}(U_0)$. As $U$ is connected, one may find $\ell_0$ such that $U_{\ell_0}\cap U_0\neq\emptyset$. Then, $s=s_{\ell_0}\in\mathrm{Dom}(U_0)\cap\mathrm{Dom}(U_{\ell_0})$ and $\hat{\mathfrak{g}}_{|U_0\cup U_{\ell_0}}=\mathfrak{g}_{U_0\cup U_{\ell_0}}(s)$, a contradiction that concludes the proof.
\end{proof}

The existence of Lie--Bäcklund compatible vector fields that are also integrable is not guaranteed in general. We shall go more in depth on this point in the sequel.  

When there are such vector fields, notice that they do not necessarily define a distribution, as example~\ref{ex::not_stable_by_addition} shows.
Therefore the set of Lie-Bäcklund compatible and integrable vector fields on a diffiety $\mathfrak{X}$, that we denote $\mathcal{D}(\mathfrak{X})$, may or not be empty, may or not be a distribution, and in the latter case, it may or not have constant rank, and may or not be involutive.

For each $\updelta \in \mathcal{D}(\mathfrak{X})$, theorem~\ref{thm::pseudo_group_flot} shows that such a derivation defines a pseudogroup of local Lie-Bäcklund automorphisms, that we have denoted $\mathfrak{G}_\updelta$. These pseudogroups, and the way to compute them, are the core subject of this article. 

Finding a vector field $\updelta$ that commutes with $\uptau$ is computationally straightforward. On the other hand testing integrability can be more difficult and proving that there is no nonzero integrable Lie-Bäcklund compatible vector field remains an open problem in the general case. In this work, we only introduce a general method for symmetries of a given order, but it is computationally difficult, so that we limit ourselves to a collection of basic examples that could be solved using simplifications. 

\section{Some theoretical results}\label{sec::}

\subsection{Order of an integrable symmetry and an analog of Sluis -- Rouchon theorem}\label{subsec::Order_Sluis_Rouchon}
In this section we introduce results that allow in some cases to bound the order of an integrable symmetry. We will start with the easy case of one control before considering multi-input systems. We need first a technical lemma.

\begin{lemma}
\label{lem::e_and_flow}
Consider a diffiety $\mathfrak{X}$, which Cartan field is $\uptau$. Let $\updelta$ be a vector field that commutes with $\uptau$ and which is integrable. Then by theorem~\ref{thm::pseudo_group_flot}, $\updelta$ defines a one-parameter pseudogroup of local Lie-B\"acklund automorphisms. Moreover let $e = \max_{i} \max_k \ord_\uptau \, \updelta^k x_i$. Assume that $e=\max_k \ord_\uptau \, \updelta^k x_{i_0}$. Then, we have $e = \max_{s\in I(U)} \ord_\uptau \, g_{x_i}(s)$, where $U$ is any open subset of $\mathfrak{U}$ such that $\overline{U}$ is compact.   
\end{lemma}
\begin{proof}
The equality $e = \max_i \max_s \ord_\uptau \, g_{x_{i_0}}(s)$ follows from the fact that by definition the successive derivatives $\partial_s^k g_{x_{i_0}}(s)$ are precisely equal to the sequence $\updelta^k x_{i_0}$, $k\in\mathbb{N}$. Proceeding as in the proof of th.~\ref{thm::pseudo_group_flot}~ii)~a), the interior of $I(U)$ is nonempty if $U$ is any open subset of $\mathfrak{U}$ such that $\overline{U}$ is compact, so that we can consider successive derivatives until one is equal to $\updelta^{k_0} x_{i_0}$ and of order $e$, which implies that $g_{x_{i_0}}(s)$ is of order $e$.
\end{proof}

\begin{theorem}\label{thm::ord_one_control}
    Let $\updelta$ be an integrable infinitesimal symmetry of a diffiety defined by a system~\eqref{eq::system} with one control. Then, for all $1\le i\le n$, $\delta x_i$ does not depend on the control $u$ nor any of its derivatives $u^{(k)}$.
\end{theorem}
\begin{proof}
    Without loss of generality, we can reduce to the case $u=\dot x_{i_0}$, for any state variable $x_{i_0}$ such that all derivatives $x_{i_0}^{(k)}$, $k\in\mathbb{N}$ are functionally  independent. Assume then that $\ord_{x_{i_0}}\updelta x_{i_0}=r>0$. We would have $\ord\updelta^k x_{i_0}=kr$, which contradicts integrability. Now, if the derivatives $x_{i_1}^{(k)}$, $k\in\mathbb{N}$ are functionally dependent, this is also the case of $g_{x_{i_1}}(s)$, so that $g_{x_{i_1}}(s)$ cannot depend on any strict derivatives of $x_0$ and $\ord\updelta x_1=0$ by lem.~\ref{lem::e_and_flow}.
\end{proof}

The next theorem is inspired by the classical results of Sluis~\cite{Sluis1993} and Rouchon~\cite{Rouchon1994} (see also~\cite{Ollivier1998} for an approach closer to this one). 

\begin{theorem}
\label{thm::rouchon_like}
Consider two diffieties $\mathfrak{X} = X \times (\R^\infty)^m_1$ and $\mathfrak{Y}  = Y \times (\R^\infty)^m_2$. The Cartan fields are given by $\uptau_1 = \sum_{i=1}^{n_1} f_i(\overline{x},\overline{u}) \partial_{x_i} + \sum_{j=1}^{m_1} \sum_{k \in \N} u_j^{(k+1)} \partial_{u_j^{(k)}}$ and $\uptau_2 = \sum_{i=1}^{n_2} g_i(\overline{y},\overline{v})\partial_{y_i} + \sum_{j=1}^{m_2} \sum_{k \in \N} v_j^{(k+1)} \partial_{v_j^{(k)}}$. 
	
As in rem.~\ref{rem::controls} we may assume that $u_i = \dot{x}_{i}$ for $i=1, \ldots, m_1$ and $v_i = \dot{y}_{i}$ for $i=1, \ldots, m_2$. 	
Assume these diffieties are Lie-B\"acklund isomorphic. Therefore we have relations: $x_i = X_i(y, \dot y, \cdots, y_2^{(k_0)})$ for $1 \leq i \leq n$ and $k_0\in\mathbb{N}$ minimal.

Consider the derivation $D \in \mathrm{Der}_\R(\mathcal{O}(\mathfrak{X})[A_1, \cdots, A_{m_1}])$ (where $A_1, \cdots, A_{m_1}$ are indeterminates), such that (i) $D =\sum_{i=1}^{m} A_i \partial_{\dot x_i}$ and we have (ii) $DA_i = 0$. 

If there exist $(i_0,i_1)\in[1,n_1]^2$ such that $\partial_{x_{i_0}^{(k_0)}}X_{i_1}\neq0$, then the algebraic and homogeneous ideal $\mathcal{J}_D$ in the $A_i$ $\mathcal{I} = (D^{\ell} (\dot{x}_i - f_i(\overline{x},\overline{u})))_{1 \leq i \leq n_1,\> \ell \in \N}$ has a non-trivial solution in the algebraic closure of the domain $\mathcal{O}(\mathfrak{X})$.
\end{theorem}
\begin{proof}
We can specialize $A_i$ to $\hat A_i = \frac{\partial X_i}{\partial y_{i_0}^{(k_0+1)}}$, considering this function that belongs to $\mathcal{O}(\mathfrak{Y})$, as a function of $\mathcal{O}(\mathfrak{X})$ using the isomorphism. 

Moreover,  with $\hat D:=\sum_{i=1}^{m_1} \hat A_i\partial_{\dot x_i}$, we have $\hat{D} \hat A_i = 0$, since $X_i$ is of order $k>0$ in  $y_{i_0}$, so that $\hat A_i = \frac{\partial X_i}{\partial y_{n+j_0}^{(k)}}$ does not depend on $y_{n+j_0}^{(k+1)}$. 

We remark that, for any $1\le i\le n_1$, $\dot{x}_i - f_i(x, \dot x)$ is identically $0$, when considered as a function of $\mathcal{O}(\mathfrak{Y})$, using the isomorphism, and that for any $F(X,\dot X)$, we have: $\hat D F(X,\dot X)=\partial_{y_{i_0}^{(k_0+1)}}F(x,\dot x)$, as $\dot X_i=\sum_{j=1}^{m_2}\partial_{y_{j}^{(k_0)}}X_i y_{j}^{(k_0+1)}+$ some terms of strictly lower order. So, for any $k\in\mathbb{N}$ and any $1\le i\le n_1$, $\hat D^k(\dot{x}_i - f_i(x,\dot x))=\partial_{y_{i_0}^{(k_0+1)}}^k(\dot{X}_i - f_i(X,\dot X))=0$ and the ideal $\mathcal{J}_D$ has at least the non-trivial solution $\hat A$.
\end{proof}

\begin{corollary}\label{cor::k_e}
Consider a single diffiety $\mathfrak{X}$ with the Cartan field $\uptau$. Let $\updelta$ be a Lie-B\"acklund compatible and integrable vector field on $\mathfrak{X}$ of iterated order $e$. If $e > 0$, the ideal $\mathcal{J}_D$ has nontrivial solutions.  More precisely, if $e=\ord\updelta^{k_0}x_{i_0}$, then there exists a solution $\hat A$ such that $\hat A_{i_0}\neq0$.
\end{corollary}
\begin{proof}
The equality between the iterated order $e$ and the order $g_{x_{i_0}}(s)$ is a direct consequence of lemma~\ref{lem::e_and_flow}. We only have then to remark that $\mathfrak{g}_U\in\mathfrak{G}_{\updelta}$, for any open subset $U\subset\mathfrak{U}$ such that $\overline{U}$ is compact is defined by $x_i=g_{x_{i_0}}(s)$ to apply theorem~\ref{thm::rouchon_like}.  
\end{proof}

\begin{example}
    Consider the equation $P=0$ with $P:=\dot x_3-\dot x_1\dot x_2$. We have $D:=\sum_{i=1}^3 A_i\partial_{\dot{x}_i}$, so that $D^2P=2 A_1A_2=0$, meaning that $A_1$ or $A_2$ must be $0$. Assume that $\updelta$ is an integrable infinitesimal symmetry of order $k>0$ of the diffiety defined by this equation, this corollary implies that $\ord\updelta x_1$ and $\ord\updelta x_2$ cannot both be equal to $k$.
\end{example}

\begin{example}\label{ex::2_2}
    Consider the equation $P=0$ with $P:=\dot x_3-\dot x_1^2\dot x_2^2$. We have again $D^4P=A_1^2A_2^2=0$, meaning that $A_1$ or $A_2$ must be $0$. Assume $A_1=0$, then $D^2P= -2 A_2^2 \dot{x}_1^2 =0$, so that $A_2$ is also $0$. (By symmetry, we also have $A_1=0$ assuming $A_2=0$). So, any integrable infinitesimal symmetry of the diffiety defined by $P$ must be of order $0$.
\end{example}

\subsection{PDE structure defined by an infinitesimal symmetry}\label{subsec::PDE}

A new diffiety structure $\hat V$ is then defined by considering the two commuting derivations: $\tau$ and $\updelta$. Using the coordinates $x_i$, for $1\le i\le n$ and $x_i^{(k)}$, for $1\le i\le m$ and $k\in\mathbb{N}^{\ast}$,  one just have to consider the expression of $\updelta$ given by
\begin{equation}\label{eq::delta}
\updelta = \sum_{i=1}^{m
}\sum_{k\in\mathbb{N}} a_{i}^{(k)}\partial_{x_i^{(k)}} + \sum_{i=m+1}^n a_i\partial_{x_i}.
\end{equation}

So, the diffiety $\hat V$ is the subdiffiety of $\mathbb{T}_{\updelta,\tau}^n$ defined by the differential ideal $(\updelta x_i-a_i,\>1\le i\le n;\> \tau x_i-f_i,\> m<i\le n)$ in the ring $\mathcal{O}(\mathbb{T}_{\updelta,\tau})$ of analytic functions on open subsets of $\mathbb{T}$ containing $\hat V$ and depending on a finite number of derivatives of the $x_i$ for the two derivations $\updelta$ and $\tau$. 

We will need to extend the notion of characteristic set~\cite{Boulier2009}, coming from differential algebra~\cite{Ritt1950}, to the analytic case, using classical results about admissible orderings and Gröbner or standard bases in rings of power series. See \textit{e.g.} Galligo~\cite{Galligo1979} or Becker~\cite{Becker1990}. We will limit here to a theoretical approach. Most significant examples are algebraic or could be easily reduced to the algebraic case, for which we can rely on algorithms for computing characteristic sets~\cite{Boulier2009}.

We consider a diffiety that is an extension of $\R_{t_1, \ldots, t_r}^r$, defined as a subdiffiety of the jet space $\mathrm{J}(\mathbb{R},\mathbb{R}^n):=\mathbb{R}\times\mathbb{T}^n$ (ex.~\ref{ex::trivial} and~\ref{ex::jet}) by a differential ideal $\mathcal{I}$ in the ring $\mathcal{R}:=\mathcal{O}(\mathrm{J}(\mathbb{R},\mathbb{R}^n))$ in some neighborhood of a point $\eta$. With our assumptions, $\mathcal{I}\cap\mathcal{O}(\R_{t_1, \ldots, t_r}^r)=0$.

\begin{definition}\label{def::char_set}

We denote by $\Theta$ the commutative monoid generated by $\Delta=\{\updelta_1, \ldots, \updelta_r\}$ and $\Upsilon:=\{x_1, \ldots, x_n\}\time\Theta$ the set of derivatives. An admissible ordering $\prec$ on $\Upsilon$ is such that for all $(\delta,\upsilon)\in\Delta\times\Upsilon$, $\delta\upsilon\succ\upsilon$ and for all $(\delta,\upsilon_1,\upsilon_2)\in\Delta\times\Upsilon^2$, $\delta\upsilon_1\prec\delta\upsilon_2$ is equivalent to $\upsilon_1\prec\epsilon_2$.

The \emph{main derivative} $\upsilon_P$ of $P\in\mathcal{R}$ is the greatest derivative appearing in $R$, \textit{i.e.} such that $\partial_{\upsilon_P}P$ is nonzero.

We consider then $P$ as a power series in $\upsilon-\eta_{\upsilon_P}$, where $\eta_{\upsilon_P}$ is the value of $\upsilon$ at point $\eta$, with coefficients in $\mathcal{R}$, not depending on $\upsilon_{P}$: $\sum_{k\in\N} c_k\upsilon_P^k$. The degree $d_P$ is the smallest index $k$ such that $c_k\neq0$. The rank of $P$ is $\mathrm{rank}P:=\upsilon_P^{d_P}$. We extend $\prec$ to $\mathcal{R}$ by comparing ranks: $P_1\prec P_2$ if $\upsilon_{P_1}\prec\upsilon_{P_2}$ or if $\upsilon_{P_1}=\upsilon_{P_2}$  and $d_{P_1}<d_{P_2}$. 

The \emph{separant} of $P$ is $S_P:=\partial_{\upsilon_P}P$ and the \emph{initial} of $P$ is the coefficient of $\upsilon_P^{d_P}$, denoted by $I_P$.

A \emph{chain} $\mathcal{A}$ of $\mathcal{I}$ is a subset $\{A_1, \ldots, A_q\}$ of $\mathcal{I}$, such that for any $A\in\mathcal{A}$ $I_A\notin\mathcal{I}$, for any couple $(A_i,A_j)\in \mathcal{A}^2$, $\upsilon_{A_i}$ is not equal to or a derivative of $\upsilon_{A_j}$. Assuming that the $A_i\in\mathcal{A}$ appear in increasing order, we extend the ordering $\prec$ to the set of chains in the following way: Let $\mathcal{A}_i$, for $i=1,2$ be two chains $\mathcal{A}_i=\{A_{i,1}, \ldots, A_{i,q_1}\}$, we have $\mathcal{A_1}\prec\mathcal{A}_2$ if there exists $1\le j\le\min(q_1,q_2)$ such that $\mathrm{rank}A_{1,\ell}=\mathrm{rank}A_{2,\ell}$ for any $1\le \ell<j$ and $\mathrm{rank}A_{1,j}<\mathrm{rank}A_{2,j}$ or if $q_1>q_2$ and 
$\mathrm{rank}A_{1,\ell}=\mathrm{rank}A_{2,\ell}$ for any $1\le \ell\le q_2$. 

By Kolchin~\cite[prop.~3 p.~81]{Kolchin1973}, $\prec$ is a well order on chains. A \emph{characteristic set} of $\mathcal{I}$ is a chain of $\mathcal{I}$ of minimal rank.
\end{definition}

\begin{proposition}\label{prop::W_gen} If $\mathcal{I}$ is a prime ideal which is equal to the ideal of functions in $\mathcal{O}(J(\R^r,\R^n))$ that vanish on all zeros of $\mathcal{I}$, then for any admissible ordering $\prec$, we have the two following properties.

    i) For any characteristic set $\mathcal{A}$ for $\mathcal{I}$ with respect to $\prec$ and any $A\in\mathcal{A}$, we have $S_A\notin\mathcal{I}$.

    ii) For any diffiety $V$ locally defined defined by the ideal $\mathcal{I}$, there exists a dense open set $W\subset V$ such that any point $\eta\in W$ admits a characteristic set $\mathcal{A}$ for $\prec$ containing only elements $A\in\mathcal{A}$ of the form $\upsilon_A-R$, where $R$ does not depend on $\upsilon_A$.
\end{proposition}
\begin{proof}
    i) Assume that $\mathrm{Sep}A\notin\mathcal{I}$. If $d_A>1$, then $d_{S_A}=d_A-1$ and one would get a strictly smaller autoreduced subset of $\mathcal{I}$ by replacing $A$ by $\mathrm{Sep}A$, in $\mathcal{A}$, which would contradict the minimality of $\mathcal{A}$. If $d_A=1$, then we have $I_P=0$ at all points of the set of zeros of $\mathcal{I}$. This is impossible, as by the definition of chains, $I_P\notin mathcal{I}$ and according to our hypotheses all functions that vanish on the set of zeros of $\mathcal{I}$ be long to $\mathcal{I}$.

    ii) At any point $\eta$ where $\mathrm{Sep}A$ does not vanish, for any characteristic set $\mathcal A$ defined at some other point $\hat\eta$, one may apply the implicit function theorem and get a $P$ in $\mathcal{I}$ with $\upsilon_P=\upsilon_A$ of the form $\upsilon-R$, where $R$ does not depend on $\upsilon_A$. So, on the open set where all separants of $\mathcal{A}$ do not vanish, there exists a characteristic set of the requested form.
\end{proof}

If the differential ideal $\mathcal{I}$ is generated by equations associated to~\eqref{eq::sysbis} and 
\begin{equation}\label{eq::sys_delta}
    \updelta x_i - a_i(x),
\end{equation}
where $a_i$ is differential in $\uptau$ but does not depend on any partial derivative involving $\updelta$, then $\mathcal{I}$ is such that any expression vanishing on the zeros of $\mathcal{I}$ belongs to $\mathcal{I}$ and we can freely apply the conclusions of prop.~\ref{prop::W_gen}.

\subsection{Structure of integrable vector fields}\label{subsec::structure}

After those preliminaries, we can go back to the case of an ordinary diffiety $V_\uptau$ that we turn into a partial diffiety $V_{\uptau,\updelta}$, where $\updelta$ is an infinitesimal symmetry, meaning that it commutes with $\uptau$. On may remark that the system~\eqref{eq::sysbis} defining the diffiety, together with the equations~\eqref{eq::delta} defining $\updelta$ form a characteristic set of the ideal defining $V_{\uptau,\updelta}$ for any admissible ordering $\prec_\uptau$ such that $\theta_1 x_i\prec_\uptau \theta_2 x_i$ for any $1\le i\le n$, if $\theta_1<\theta 2$ for the lexicographic ordering with $\uptau>\updelta$. This means that we write $\theta_i=\uptau^{\alpha_i}\updelta^{\beta_i}$ and say that $\theta_1<\theta_2$ if $\alpha_1<\alpha_2$ or if $\alpha_1=\alpha_2$ and $\beta_1<\beta_2$. 

We will need to work with a different ordering $\prec$.

By prop.~\ref{prop::W_gen}, we may assume that we are working at some point $\eta$ where all functions in the characteristic set $\mathcal{A}$ for $\prec_\updelta$ are linear in their main derivatives. We may assume that for any $1\le i\le n$, the $h_i$ elements of $\mathcal{A}$ of which the leading derivative is a derivative of $x_i$ are of the form $A_{i,k}$, with $\upsilon_{A_{i,k}}=\updelta^{\beta_{i,k}}\uptau^{\alpha_{i,k}}x_i$. We assume that the $\beta_{i,k}$ are decreasing with $k$, and so that the $\alpha_{i,k}$ are increasing.

\begin{proposition}\label{prop::coord_delta}
    With the above hypotheses, the set $\mathrm{Stair}\mathcal{A}$ of derivatives that are irreducible by $\mathcal{A}$ form a set of independent coordinates for $V_{\uptau,\updelta}$ in a neighborhood of the point $\eta$.
\end{proposition}
\begin{proof}
    Assume that there is a relation $P$ between derivatives in this set. Then, $P\in\mathcal{I}(V)$ and is irreducible by $\mathcal{A}$, so that we can construct a subset of $\mathcal{I}(V)$ smaller that $\mathcal{A}$. A contradiction to the minimality of $\mathcal{A}$. So this set is independent. On the other hand, this is a maximal independent set, as any derivative $\upsilon\in\Upsilon\setminus\mathrm{Stair}\mathcal{A}$ is  elementarily reducible by some derivative $\theta A$ of some $A\in\mathcal{A}$, so that it satifies the relation $\theta A(\upsilon, \ldots) =0$.
\end{proof}

We first define $e_i$ to be $\beta_{i,1}$ and $f_i$ to be $\alpha_{i,h_i}$ and write $A_{i,k}=\updelta^{\beta_{i,k}}\tau^{\alpha_{i,k}}x_i-F_{i,k}(\mathrm{Stair}_{\mathcal{A}})$, for $1\le i\le n$ and $1\le k\le h_i$. Using those coordinates and the conventions $\alpha_{i,0}=0$, $\alpha_{i,h_i+1}=\infty$, $\beta_{i,0}=\infty$ and $\beta_{i,h_i+1}=0$, one may now express $\tau$ and $\updelta$ in the following way.:
\begin{equation}\label{eq::deltabis}
\updelta=\sum_{i=1}^{n}\left(
\sum_{\kappa=0}^{h_i}\sum_{\mu=\alpha_{i,\kappa}}^{\alpha_{i,\kappa+1}-1} 
\left(
\uptau^{\mu-\alpha_{i,\kappa}} F_{i,\kappa}\partial_{\uptau^{\mu}\updelta^{\beta_{i,\kappa}-1}x_i}+
\sum_{\lambda=0}^{\beta_{i,\kappa}-2} \uptau^{\mu}\updelta^{\lambda+1}x_i\partial{\uptau^{\mu}\updelta^{\lambda}x_i}
\right)
\right)
\end{equation}
and
\begin{equation}\label{eq::taubis}
\uptau=\sum_{i=1}^{n}\left(
\sum_{\kappa=1}^{h_i+1}\sum_{\nu=\beta_{i,\kappa}}^{\beta_{i,\kappa-1}-1} 
\left(
\updelta^{\nu-\beta_{i,\kappa}} F_{i,\kappa}\partial_{\uptau^{\alpha_{i,\kappa}-1}\updelta^{\nu}x_i}+
\sum_{\ell=0}^{\alpha_{i,\kappa}-2} \uptau^{\ell+1}\updelta^{\nu}x_i\partial{\uptau^{\ell}\updelta^{\nu}x_i}
\right)
\right).
\end{equation}
This last expression defines a morphism from $V$ defined using the ordering $\prec_\updelta$ to the previously used coordinates using $\prec_\uptau$. 

\begin{remark}
    One must be careful that $\updelta$ appears in the expression of $\uptau$ that appears itself in that of $\updelta$. One may check that it is consistent, as the $\updelta$ (resp.~$\uptau$) derivatives appearing in the expression of $\uptau\upsilon$ (resp.~$\updelta\upsilon$), for any derivative $\upsilon$ are strictly smaller than $\upsilon$ for the chosen ordering, and admissible orderings are known to be well-orders. See Cox \textit{et al.} for details~\cite{Cox1997}.
\end{remark}

In the set $\mathrm{Stair}\mathcal{A}$ one may distinguish two main classes of infinite subsets.

\begin{definition}\label{def::chimney_tunnel}
    The subset of derivatives $\uptau^k \updelta^{\ell}x_i$$1\le i\le n$, $k\in\N$ and $0\le\ell<\beta_{i,h_i}$ is the \emph{chimney} of $\mathrm{Stair}(\mathcal{A})$. 

    The subset of derivatives $\uptau^k \updelta^{\ell}x_i$, for $1\le i\le n$, $\ell\in\N$ and $0\le k<\alpha_{i,1}$ is the \emph{tunnel} of $\mathrm{Stair}(\mathcal{A})$. See fig.~\ref{fig::chimney_tunnel}.
\end{definition}

\begin{figure}[H]
    \centering
    \includegraphics[width=0.40\linewidth]{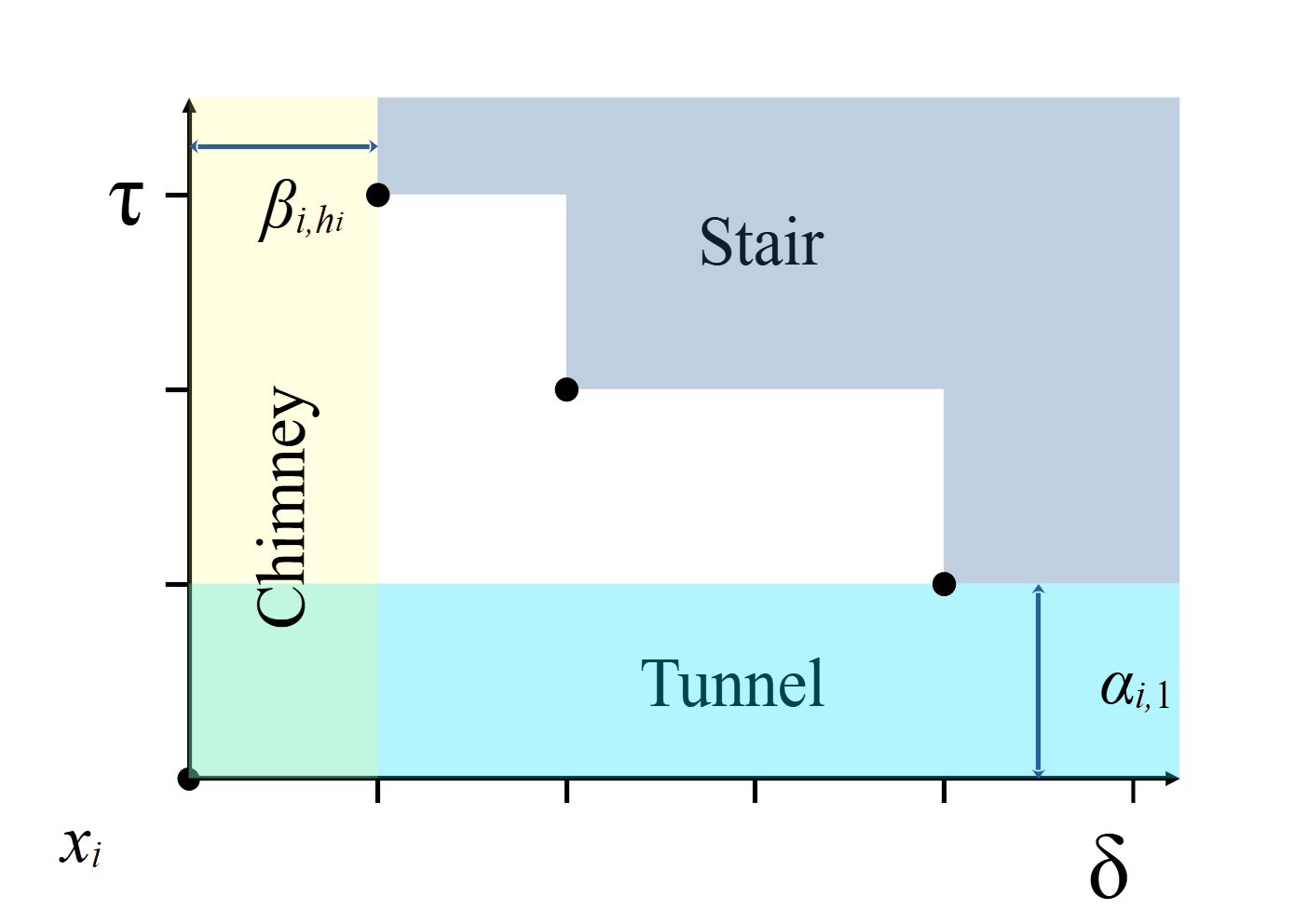}
    \caption{Chimney and tunnel for derivatives of $x_i$; main derivatives of elements of $\mathcal{A}$ are indicated with {\large$\bullet$}.}
    \label{fig::chimney_tunnel}
\end{figure}

Using an admissible ordering $\prec$ and considering a characteristic set $\mathcal{A}$ for this ordering, one may express iterated derivatives with respect to $\uptau$ in the coordinate functions of $\mathrm{stair}(\mathcal{A})$. Using those coordinates, the order of $\uptau^k x_i$ may be different from $k$. 

\begin{example}\label{ex::tau_order}
    Consider in $\mathbb{T}^2$, $\updelta=\sum_{k\in\mathbb{N}} x_1^{(k)}\partial_{x_1^{(k)}}+\sum_{k\in\mathbb{N}} x_i^{(k+2)}\partial_{x_2^{(k)}}$. For a $\updelta$-ordering $\prec$ with all derivatives of $x_2$ smaller than all derivatives of $x_1$, a characteristic set of $\mathbb{T}_{\uptau,\updelta}^2$ is $\{\uptau^2 x_1 -\updelta x_2, \updelta x_1 - x_1, \updelta^2 x_2 -\updelta x_2\}$. In this representation, the order of $\uptau^2 x_1=\updelta x_2$ is $0$.
\end{example}

The following lemma will be useful.

\begin{lemma}\label{lem::tau_order}
Let $\mathcal{A}$ be a characteristic set with respect to an arbitrary ordering.
    We denote by $\ord_{\mathcal{A},\uptau}$ the order with respect to the coordinates of $\mathrm{stair}(\mathcal{A})$, assuming that all functions of $\mathcal{A}$ are linear with respect to their leading derivatives and can be written $A=\upsilon_A - R_A$.

    We have then $\ord_{\mathcal{A},\uptau} \uptau^k x_i\le k+\max_{A\in\mathcal{A}} \ord_\uptau R_A-1$, for $1\le i\le n$.
\end{lemma}
\begin{proof}
    It is enough to remark that, using \eqref{eq::taubis}, if $\ord_{\mathcal{A},\uptau} \uptau^k x_i<\max_{A\in\mathcal{A}} \ord_\uptau \upsilon_A$, then $\ord_{\mathcal{A},\uptau} \uptau^{k+1} x_i\le \max_{A\in\mathcal{A}} \ord_\uptau R_A$ and if $\ord_{\mathcal{A},\uptau} \uptau^k x_i\ge\max_{A\in\mathcal{A}} \ord_\uptau \upsilon_A$, then $\ord_{\mathcal{A},\uptau} \uptau^{k+1} x_i=\ord_{\mathcal{A},\uptau} \uptau^k x_i +1$.
\end{proof}

We need now to focus on a particular class of orderings.

\begin{definition}\label{def::order_delta}
    We say that an admissible ordering $\prec$ is a $\delta$-ordering
    if it is such that $\theta_1 x_i\prec\theta_2 x_i$ for any $1\le
    i\le n$, if $\theta_1<\theta_2$ for the lexicographic ordering
    with $\updelta>\uptau$, \textit{i.e.} such that
    $\theta_1=\updelta^{\beta_1}\uptau^{\alpha_1}<\theta_2=\updelta^{\beta_2}\uptau^{\alpha_2}$
    if $\beta_1<\beta_2$ or if $\beta_1=\beta_2$ and
    $\alpha_1<\alpha_2$.  
\end{definition}

We can now prove the following proposition.

\begin{theorem}\label{th::tunnel}
Still using the same definitions, we call $\hat c:=\sum_{i=1}^{n}\alpha_{i,1}$ the tunnel number and  $\hat t:=\sum_{i=1}^{n}\beta_{i,h_i}$ the chimney number.

i) For any ordering $\prec$, the chimney number is at most $m$.

ii) Using a $\updelta$-ordering $\prec_\updelta$, the following three propositions are equivalent:

{\parindent =1cm

a) The derivation $\updelta$ is integrable; 

b) The tunnel number is equal to $0$;

c) The chimney number is equal to $m$.

}    
\end{theorem}
\begin{proof} i) Let $\hat e$ be the maximal order of the $\updelta^k x_i$, for $0\le k<b_{i,h_i}$ and all $1\le i\le n$. There are $\hat c(r+1)$ independent functions $\uptau^k \updelta^{\ell}x_i$ for $0\le k\le r$ and $0\le\ell<\beta_{i,h_i}$ in the chimney, that are of order at most $\hat e+r$ in $\uptau$, so depend on at most $(\hat e+r)m+n$ $\uptau$-derivatives of the $x_i$, which imply 
\begin{equation}\label{eq::c_m}
    \hat c(r+1)\le (\hat e+r)m+n,
\end{equation}
so that $\hat c\le m$.

    ii) The equivalence of a) and b) is a straightforward consequence of definition~\ref{def::integrability}. Indeed, as their ordering is bounded, the derivatives $\updelta^k x_i$, for all $1\le i\le n$, $k\in\mathbb{N}$ cannot be independent, which is a) $\Leftrightarrow$ b). Reciprocally, if they are not independent, there exists $k_0$ such that $\updelta^{k_0} x_i$ is a function of the $\updelta^k x_i$, for $k<k_0$ and $\max_{k\in\mathbb{N}}\mathrm{ord}\updelta^k x_i=\max_{k<k_0}\mathrm{ord}\updelta^k x_i$ and so is bounded.

    By i), we know that $\hat c\le m$. Let $b$ be cardinal of the subset $S_2$ of elements of $\mathrm{stair}(\mathcal{A})$ not in the chimney. Let us first prove that if $\hat c=m$, then $b<\infty$, which means c) $\Rightarrow$ b). We proceed like for \eqref{eq::c_m} in the proof of i), denoting by $b_r$ the number of independent functions of order at most $\hat e+r$ in $S_2$. We have 
    \begin{equation}\label{eq::c_m2}
    m(r+1)+b_r\le (\hat e+r)m+n\>\Rightarrow\> b_r\le (\hat e-1)m+n,
\end{equation}
for all $r\in\mathbb{N}$, so that $b\le (\hat e-1)m+n$.
    To prove that b) $\Rightarrow$ c), we need prove that $b<\infty$ $\Rightarrow$ $\hat c=m$. By lem.~\ref{lem::tau_order}, the $rm+n$ independent $\tau$-derivatives of order at most $r$ are functions of at most $\hat c (r+a)+b$ independent functions, where $a:=\max_{A\in\mathcal{A}} \ord_\uptau R_A-1$. We have then 
    $$
    rm+n\le \hat c (r+a)+b,
    $$
    so that $\hat c\ge m$.
\end{proof}

\begin{example}\label{ex::tunnel1} We consider the infinitesimal symmetry $\updelta=\sum_{k\in\mathbb{N}} x^{(k+2)}\partial_{x^{(k)}}$ of $\mathbb{T}^1$. It is not integrable, as its chimney number is $0$ and its tunnel number is $2$. Indeed, for $\prec_{\updelta}$ its characteristic set is $\tau^2x=\updelta x$. See fig.~\ref{fig::tunnel1}.

\begin{figure}[H]
    \centering
    \includegraphics[width=0.25\linewidth]{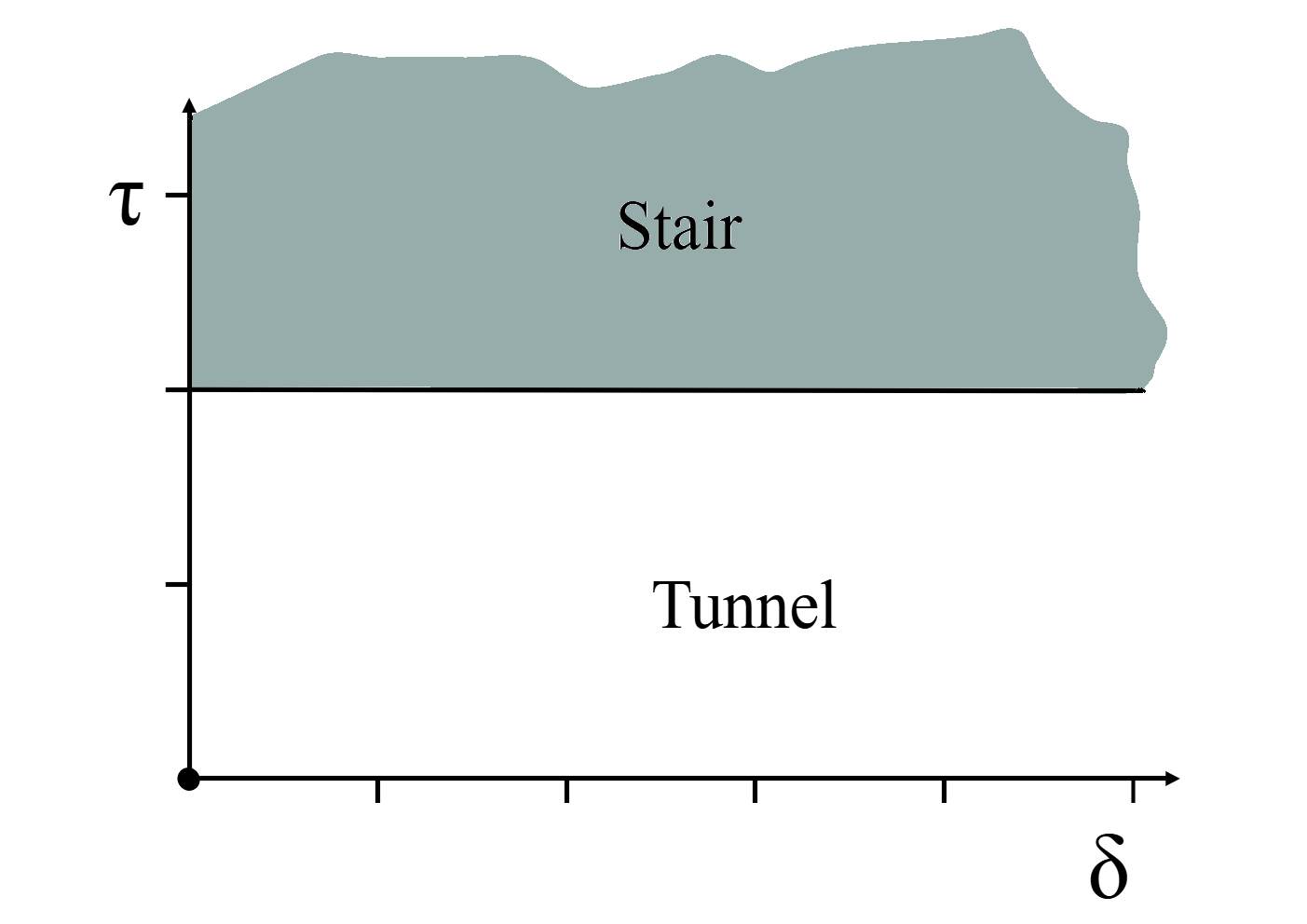}
    \caption{Tunnel and stair from ex.~\ref{ex::tunnel1}}
    \label{fig::tunnel1}
\end{figure}
\end{example}

\begin{example}\label{ex::chimney} Taking the infinitesimal symmetry defined by $\updelta x_1 = x_1+x_2''$ and $\updelta x_2= x_2$, a characteristic set for a $\delta$-ordering with all derivatives of $x_2$  greater that all derivatives of $x_1$ is $\{\updelta^2 x_1=2\updelta x_1 - x_1, \updelta x_2=x_2, \uptau^2 x_2=\updelta x_1 - x_1\}$, so that the tunnel number is $0$ and the chimney number is $2$: $\updelta$ is integrable.  
See fig.~\ref{fig::chimney}.
\end{example}

\begin{figure}[H]
    \centering
    \includegraphics[width=0.25\linewidth]{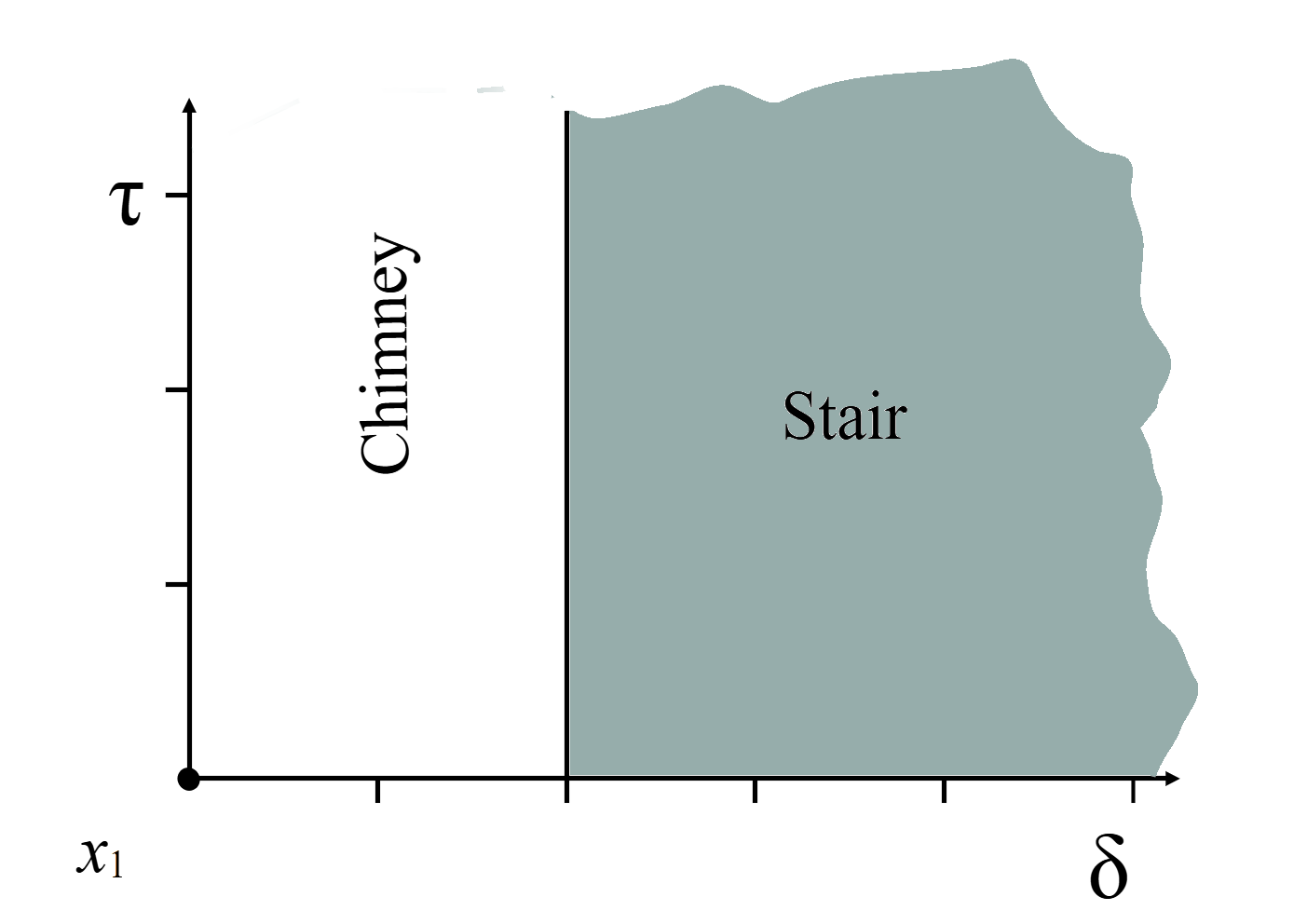}\hskip 1cm
    \includegraphics[width=0.25\linewidth]{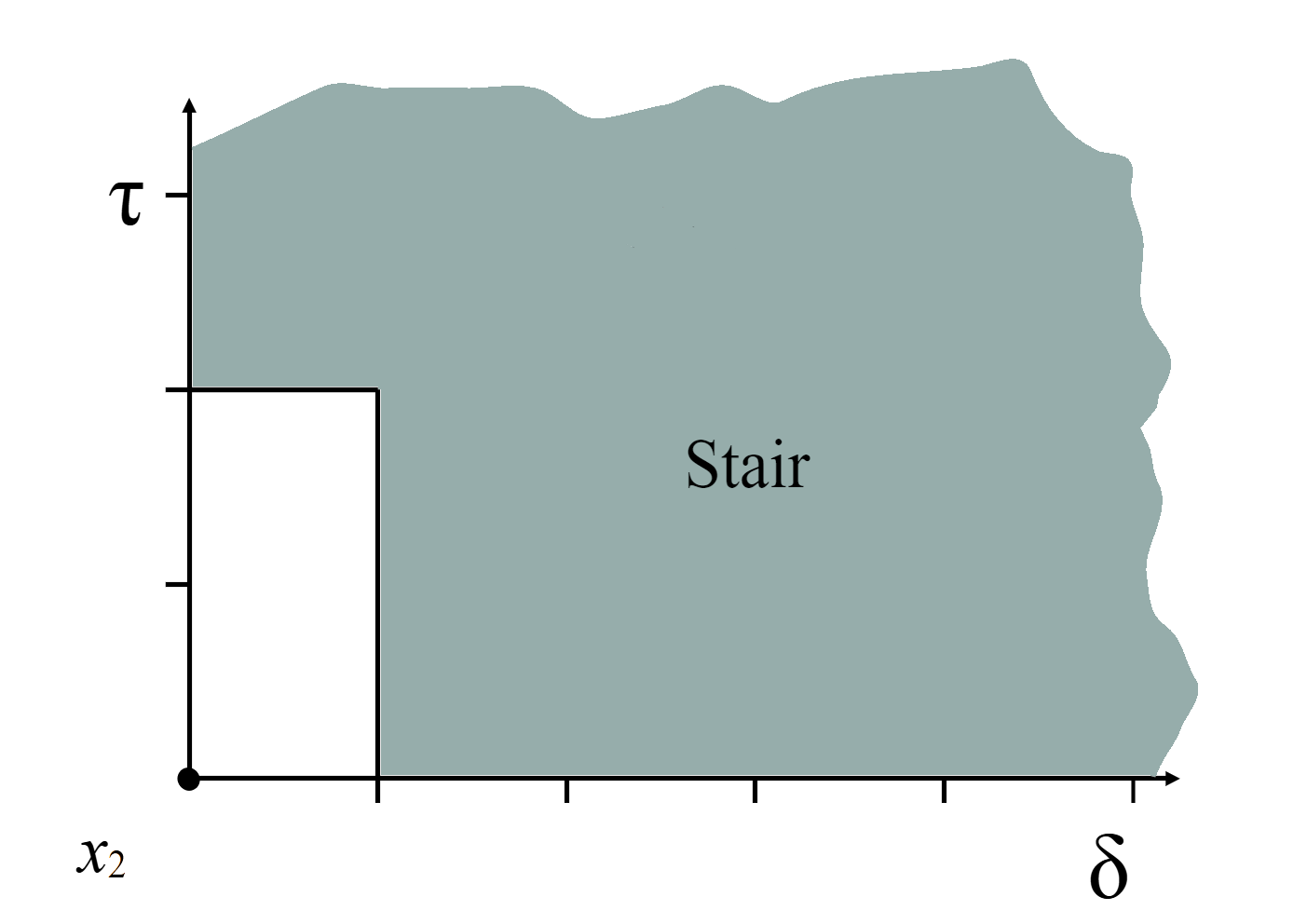}
    \caption{Chimney and stair from ex.~\ref{ex::chimney}}
    \label{fig::chimney}
\end{figure}

Using the next theorem, we will be able to bound the orders $r_i$ in lem.~\ref{lem::integrability}. For its proof, we will need to use Jacobi's bound. See \cite{FO-BS-07} for the diffiety setting and \cite{Ollivier2023} for more details in the algebraic case.

\begin{theorem}\label{th::bound_ord}
Using the same notations as above, assume that $\updelta$ is an integrable symmetry such that $\updelta x_i=a_i$ and that $\max_{i=1}^{n}\ord_\uptau a_i=\varpi$. Then, for all $1\le i\le n$, we have 
$$
\left(\sum_{i=1}^{n}\beta_{i,1}\right)\le\frac{n(n+1)}{2}\varpi+n-m.
$$    
\end{theorem}
\begin{proof} By th.~\ref{th::tunnel}, if $\updelta$ is integrable, then the chimney number $\sum_{i=1}^{n}\beta_{i,h_i}$ is equal to $m$. This means that the $\updelta^k x_i$ for $1\le i\le n$ and $0\le k<\beta_{i,h_i}$ and all their $\uptau$ derivatives are functionally independent. Let $F_{i,k}$ be the expression of $\updelta^k x_i$ as a function of $\uptau$ derivatives of the $x_i$. We consider the $\uptau$ differential system 
$$
F_{i,k}=\updelta^k x_i,\> 1\le i\le n,\> 0\le k<\beta_{i,h_i},
$$   
completed with the initial system~\eqref{eq::sysbis}, as a system in the variables $x_i$. This system defines the graph of the application $\uptau^\kappa \updelta^k x_i\mapsto \uptau^\kappa F_{i,k}$, so it generates a prime ideal. In the neighborhood of any point in a dense open set, there exists a characteristic set $\mathcal{B}=\{x_i^{(r_i)}-G_i|1\le i\le m\}\cup\{\dot x_i-f_i\}$ of the ideal generated by this system, where $G$ depends of the $\uptau$ derivatives of the $\updelta^k x_i$ for $1\le i\le n$ and $0\le k<\beta_{i,h_i}$ and of the $\uptau^{\kappa}x_i$, with $0\le \kappa<r_i$, the remaining elements corresponding to~\eqref{eq::sysbis}. The functional independence of the $\updelta^k x_i$ and all their $\uptau$ derivatives implies that this system is quasi-regular~\cite[def.~4.1 (0.1)]{FO-BS-07}. So, using Jacobi's bound~\cite[th.~4.3 (0.3)]{FO-BS-07}\footnote{Due to lack of knowledge, we cannot use the full power of Jacobi's bound and need to retreat to a simpler consequence of it, which is closer to Ritt's ``differential analog of Bézout's bound''~\cite[p.~135]{Ritt1950}.}, its order which is $n-m+\sum_{i=1}^m r_i$ is at most the sum of the orders of the equations, \textit{i.e.} at most $n(n+1)\varpi/2+n-m$. Indeed, $\ord_{\uptau}F_{i,k}\le k\varpi$. This means that $\sum_{i=1}^n r_i\le n(n+1)\varpi/2+n-m$.

Recursively replacing $x_i^{(r_i+\ell)}$ by $G_i^{(\ell)}$, we may express the derivatives $\upsilon$ in the set $E:=\mathrm{Stair}(\mathcal{A})\setminus\mathrm{Chimney}(\mathcal{A})$, as functions $Q_{\upsilon}$ of the derivatives in $\mathrm{Chimney}(\mathcal{A})$ and of the $x_i^{(\kappa)}$, for $0\le\kappa<r_i$: $Q_{\upsilon}=\upsilon$. Assume that $\sharp E>\sum_{i=1}^n r_i$, then we can eliminate them and get a function in $\mathrm{I}$ with a leading derivative in $\mathrm{Stair}(\mathcal{A})$, which would contradict the minimality of $\mathcal{A}$.

This implies that 
$$
\sum_{i=1}^{n}\beta_{i,1}\le\sum_{i=1}^n r_i\le\frac{n(n+1)}{2}\varpi+n-m.
$$
\end{proof}

This theorem allows to express integrability using partial differential equations.

\begin{corollary}\label{cor::int_sys}
    A symmetry defined by $\updelta x_i=a_i$, where the $a_i$ depend on derivatives obtained by iterations of $\uptau$ up to order at most $\varpi$, is integrable iff for all $1\le i\le n$, the $\updelta^k x_i$, $0\le k\le n-m+n(n+1)\varpi/2$ are functionally dependent, expressed as functions of the $\Upsilon_0:=\{\uptau^{\kappa} x_j,\>1\le j\le n,\> 0\le\kappa\le q(q+1)\varpi/2\}$, with $q=n-m+x(n+1)\varpi/2$. This is expressed by the polynomial PDE system
    {
    \begin{equation}\label{eq::integ_cond}
        \left|\frac{\partial\uptau^\kappa x_i}{\partial_\upsilon}\Big| 
        \renewcommand{\arraystretch}{0.6}
        \begin{array}{l}\scriptscriptstyle 
            1\le i\le n,\\
            \scriptscriptstyle 0\le k\le n-m+x(n+1)\varpi/2 \\
             \scriptscriptstyle \upsilon\in U
        \end{array}
        \right|=0,\> 1\le i\le n,\> U\subset\Upsilon_0,\> \sharp U=n-m+x(n+1)\varpi/2+1
    \end{equation}}
\end{corollary}

We conclude this subsection with a special result in the case $m=1$, \textit{i.e.} with a single control.

\begin{theorem}\label{th::one_control}
    If $m=1$ and $\updelta$ is an integrable symmetry iff $\ord\updelta=0$.
\end{theorem}
\begin{proof}
If $\ord\updelta=0$, then for any $k\in\mathbb{N}$ and any $1\le i\le n$, $\ord\updelta^k x_i=0$, so that $\updelta$  is integrable.

    Without loss of generality, assume that $u_1=x_1$ and that $\ord \updelta x_{i_0}=e>0$, as in \eqref{eq::sysbis}. Then $\updelta \dot x_{i_0}=\sum_{i=1}^n \partial_{x_i} f_{i_0} \updelta x_i+\partial_{\dot x_1} f_{i_0} \updelta \dot x_1$, so that $\ord\updelta x_1=e$. A simple recurrence shows that $\ord \updelta^k x_1=ke$, a strictly increasing sequence, so that $\updelta$ is not integrable: a contradiction.
\end{proof}

\subsection{Tame symmetries}\label{subsec::tame}

The following proposition shows how to build integrable infinitesimal symmetries of trivial diffieties.

\begin{proposition}\label{prop::structure}
We consider the trivial diffiety $\mathbb{T}^n$, as defined in ex.~\ref{ex::trivial}. For any partition $\{x_1, \ldots, x_n\}=\bigcup_{\ell=1}^s\mathcal{X}_{\ell}$ and any functions $G_i$, $1\le i\le n$ such that, if $x_i\in\mathcal{X}_j$, $G_i$ only depends on variables in $\left(\bigcup_{k=1}^j\mathcal{X}_k\right)$ and $G_i$ of order $0$ in all variables in $\mathcal{X}_j$  the infinitesimal symmetry of $\mathbb{T}^n$ defined by
$$
\updelta= \sum_{i=1}^{n} \sum_{k\in\mathbb{N}}\uptau^{k}G_i\partial_{x_i^{(k)}}
$$
is integrable.
\end{proposition}
\begin{proof}
   This is a simple consequence of th.~\ref{th::tunnel}, using a $\updelta$-ordering such that all derivatives of $\mathcal{X}_{j_1}$ are smaller than all derivatives $\mathcal{X}_{j_2}$ if $j_1<j_2$. With such an ordering, $\updelta x_i=G_i$ is a characteristic set, so that the tunnel number is $0$ and $\updelta$ is integrable. 
\end{proof}

\begin{definition}
\label{def::tame}
We call \emph{tame} all integrable symmetries of this kind in a given system of flat coordinates.
\end{definition}

One may notice that the definition of tame integrable symmetries is very close to the definition of differential Jonquières automorphisms of $\mathbb{T}^n$, that maps $(x_1, \ldots, x_n)$ to
\begin{equation}\label{eq::Jonquieres}
    (F_1(x_1), F_2(x_1,x_2), \ldots, F_n(x_1, \ldots, x_{n-1}, x_n)),
\end{equation}
the $F_i$ being of order $0$ in $x_i$, for all $1\le i\le n$. It has been proved that not all local automorphisms can be decomposed using Jonquières morphisms and permutations~\cite{Ollivier1999}. Not all integrable fields are tame in the usual coordinates, nor any order $0$ change of coordinates, as shown by the following examples~\ref{ex::wild1} and \ref{ex::wild2}.

To test if a given integrable symmetry is tame, we may use the following easy criterion.
\begin{proposition}\label{prop::tame}
If there exists coordinates $y_i=F_i(x)$, with $\mathrm{ord}F_i=0$, such that $\updelta$ is tame in the coordinates $y$, then, the Lie algebra $L:=\mathrm{Lie}\langle[\updelta,\partial_{x_i^{(k)}}]|1\le i\le n;\> k>0\rangle<n$ is such that $\mathrm{dim} L<n$.
\end{proposition}
\begin{proof}
    All coordinates functions $y$ in $\mathcal{Y}$ are such that $\hat\delta y=0$ for all $\hat\delta\in L$, so that $L$ cannot be of full rank.
\end{proof}
Using this criterion, simple computations allow to test that some symmetries in the following examples are not tame in the chosen system of coordinates.

\begin{example}\label{ex::wild1}
    We will use the Rouchon involution~\cite{Ollivier1998}
\begin{equation}\label{eq::Rouchon}
    (x_1,x_2)\mapsto \left(y_1=\frac{\dot x_2}{\ddot x_1},y_2=\left(\frac{\dot x_2}{\ddot x_1}\right)'-\frac{\dot x_1\dot x_2}{\ddot x_1}+x_2\right),
\end{equation}
which is known not to admit any decomposition in a sequence of permutations and Jonquières automorphisms.

The integrable symmetry $\partial_{x_1}$ (resp.~$\partial_{x_2}$) is expressed by $y_1\partial_{y_2}$ (resp.~$\partial_{y_2}$) in the $y$ coordinates. So both are tame in both coordinate systems. However, the integrable symmetry $\sum_{k\in\mathbb{N}}x_1^{(k+1)}\partial_{x_2^{(k)}}$, which is tame, in the $x$ coordinates is expressed by 
$$
\sum_{k\in\mathbb{N}} \left(\frac{1}{\left(\frac{1}{\ddot y_1}\right)''}\right)^{(k)}\partial_{y_1^{(k)}}+
\sum_{k\in\mathbb{N}} \left(
\frac{y_2'}{\ddot y_1}-\frac{\left(\left(\frac{\dot y_2}{\ddot y_1}\right)'''\right)^{2}\frac{\dot y_2}{\ddot y_1}}
{\left(\left(\frac{\dot y_2}{\ddot y_1}\right)''\right)^{2}}
\right)^{(k)}\partial_{y_2^{(k)}},
$$
which is not, in the $y$ coordinates.
\end{example}

\begin{example}\label{ex::wild2}
Consider the infinitesimal symmetry $\updelta$ of $\mathbb{T}^2$ defined by $\updelta x_1=x_1-x_1^{(4)}-x_{2}^{(6)}$ and $\updelta x_2= x_2+\ddot x_1+x_2^{(4)}$. In the coordinates $y_1=x_1+\ddot x_2$, $y_2=x_2$, $\updelta y_1=y_1$ and $\updelta y_2=y_2+\ddot y_1$, so that $\updelta$ is tame and integrable. Moreover, this change of variable is Jonquières. But there is no order $0$ change of coordinates in which it is tame as there is no non constant function $F(x)$  such that $\mathrm{ord}\updelta F=0$. \end{example}

One may wonder if all integrable infinitesimal symmetries become tame using a change of variables of order great enough.
\bigskip

It is easily seen that the set of infinitesimal symmetries is stable by addition and Lie bracket. 
But, the set of \emph{integrable} symmetries has no obvious algebraic structure. As the following example shows, it is not stable by addition or Lie bracket.

\begin{example}
\label{ex::not_stable_by_addition}
    Consider the following fields $\updelta_1:=\sum_{k}x_2^{(k+1)}\partial_{x_1^{(k)}}$ and
    $\updelta_2:=\sum_{k}x_1^{(k+1)}\partial_{x_2^{(k)}}$ on the diffiety $\mathbb{T}^2$. They are both integrable infinitesimal symmetries, but $\updelta_1+\updelta_2$ is not as $(\updelta_1+\updelta_2)x_1=x_2'$, $(\updelta_1+\updelta_2)^2x_1=x_1''$, $(\updelta_1+\updelta_2)^3x_1=x_2^{(3)}$, $(\updelta_1+\updelta_2)^4x_1=x_1^{(4)}$, \dots

    Their Lie bracket is $\updelta:=\sum_{k}x_1^{(k+2)}\partial_{x_1^{(k)}}-\sum_{k}x_2^{(k+2)}\partial_{x_2^{(k)}}$, so that $\updelta^i x_1=x_1^{(2i)}$ and $\updelta$ is not integrable.
\end{example}

\subsection{Computing Lie--B\"acklund Compatible Integrable Vector Fields}
\label{sec::computations}

In this subsection, we review the basic computation tools at our disposal to look for integrable symmetries. Two distinct problems must be solved by translating them into systems of partial differential equations: we give two different ways of looking for infinitesimal symmetries: a direct approach or using a parametrization, and then investigate how to express integrability.

\subsubsection{Solving Directly The System}
\label{subsubsec::GB}

The direct approach consists only in solving directly the system~\eqref{eq::commutativity1} in which we incorporate the constraints that functions depend on the derivative of the commands up to a finite order $N$. In order to apply a computational method, one has to know $N$ or at least an upper bound of it, if such a bound is available. We refer to~\ref{susubsec::int_cond} bellow for this topic. Considering flat systems, integrable symmetries of any order can exist when $m>1$.

In order to handle systems of linear differential
equations, one may rely on Gröbner bases computations. The basic idea
is to choose an ordering $\prec$ on derivatives, compatible with the
derivation, that is such that $\upsilon_{1}\prec\upsilon_{2}$ is
equivalent to $\upsilon_{1}'\prec\upsilon_{2}'$. A Gröbner basis is a set
of generators of a module, the leading derivatives of which generates
the monoideal of leading derivatives. It is computed by reducing
equations, using Euclidean division and using the following
``completion process''.

For any couple of equations
$\upsilon_{1}-R_{1}$ and $\upsilon_{2}-R_{2}$, where $\upsilon_{1}$
and $\upsilon_{2}$ are leading derivatives, if they are derivatives of
the same variables, let $\theta_{1}$ and $\theta_{2}$ be products of
derivatives of minimal orders such that
$\theta_{1}\upsilon_{1}=\theta_{2}\upsilon_{2}$, we complete the system
with the reduction of $\theta_{1}R_{1}-\theta_{2}R_{2}$ if it is non
zero. The completion process stops when all such ``S-polynomials'' are
reduced to $0$~\cite[th.~2.7]{Moeller1992}. An implementation is available in the Gröbner package of Maple.

In the nonlinear case, one may use characteristic sets
computation~\cite{Boulier2009}. An implementation is available in the
Diffalg package of Maple.

In the sequel, we illustrate through several examples that this kind of calculation can be made with a step by step heuristics, which consists in explicit differential constraints that reduce the dependency of $\updelta$ on the dieffiety variables. 

\subsubsection{Parametrization of The Vector Fields}
\label{sec::parametrization}

In this section, we will show that the vector fields we are looking for can be parametrized by functions that are first integrals of a certain distribution. 

We define an extended system with the dynamics defined by $\uptau_2 a_i = \updelta f_i$ for $1 \leq i \leq n$, with $\updelta = \sum_{i=1}^n a_i \partial_{x_i} + \sum_{j=1}^m \sum_{k=0}^\infty b_{j,k} \partial_{u_j^{(k)}}$ is the vector field we are looking for. 

This system is associated to the following Cartan field:
\begin{equation}\label{eq::tau2bis}
\hat{\uptau} = \uptau + \uptau_2:=\sum_{i=1}^n \left (  \left ( \sum_{h=1}^n \frac{\partial f_i}{\partial x_h} a_h \right )   + \sum_{\ell=1}^m \frac{\partial f_i}{\partial u_{\ell}} b_{\ell,0}  \right )\partial_{a_{i}} + \sum_{j=1}^m \sum_{k=0}^\infty b_{j,k+1} \partial_{b_{j,k}}.
\end{equation}

Note that this system is equivalent to~\eqref{eq::linearized} and defines an isomorphic diffiety extension. 
 
We gave an abstract proof of th.~\ref{th::lin_access}, showing that it defines a flat extension. In \cite[th.~5]{KLO-20}, inspired by Jakubczyk and Respondek~\cite{Jakubczyk1980}, we have shown that the flat outputs of a similar system may be chosen to be solution of linear PDEs. As the linearized system is precisely linear, we are able here to give a complete algorithm for computing flat outputs, relying on the following proposition. 

\begin{proposition}\label{prop::algo}
Using the notation of prop.~\ref{prop::accessibility} and $\uptau_2$ as in \eqref{eq::tau2bis}, we define 
$$
\updelta_{k}=\langle\hat\uptau_2^\ell\partial_{b_{j,0}}|1\le i\le m, 0\le\ell\le k\rangle.
$$

  i) The $\updelta_k$ are involutive.

  ii) If $\updelta_{k+1}=\updelta_k$, then $\updelta_{k+\ell}=\updelta_k$ for all $\ell\in\mathbb{N}$.

  iii) We have $\updelta_{n+1}=\updelta_n$.
  
  iv) Assuming that the system~\eqref{eq::system} is strongly accessible, we further have: $\dim\updelta_n=n+m$. 

  We define $r_k=\dim\updelta_k-\dim\updelta_{k+1}$ and $s_k=s_k-s_{k+1}$. Let $q_{\ell}$, $1\le\ell\le h$, be the increasing indices for which $s_q\neq0$. We can then define\footnote{For the readers familiar to Young diagrams, the $p_j$ form a partition of $n$ conjugate to the $s_i$.} $p_{r_{q_{\ell}-1}+1}$, \dots, $p_{r_{q_{\ell}}}$ to be equal to $q_{\ell}$, for $1\le\ell\le h$.

  We define $k_i$ to be such that $r_{k_i-1}\le i\le r_{k_i}$.
  For $\ell$ from $1$ to $h$, let $\zeta_{r_{q_{\ell}-1}+1}$, $\zeta_{r_{q_{\ell}}}$ be $s_{q_{\ell}}$ independent solutions of $\updelta_{p-1}$, chosen to also be independent of the vector space $\langle\zeta_{i}^{(k)}$, for $\lambda<\ell$, $r_{q_{\ell}-1}<i\le r_{q_{\ell}}$ and $0\le k\le q_{\ell}-k_i\rangle$. 

  v) One may choose the $\zeta_i$ to be linear combinations of the $a_i$ with coefficient analytic functions in the $x_i$ and $u_{j}^{(k)}$ defined in a dense open subset of $V$.

  vi) For any $1\le i\le n$, let $k_i$ be the integer such that $r_{k_i-1}<i\le r_{k_i}$. The vector space $\langle\zeta_{i}^{(k)}|1\le\lambda\le r_1, 0\le k\le k_i\rangle$  has dimension $n+m$.

  vii) The $z_i$ form a basis of the module defined by $\uptau_2$ and also a flat output of the diffiety it defines.
\end{proposition}
\begin{proof} We may first notice that the full control condition\footnote{The Jacobian matrix $\left ( \frac{\partial f_i}{\partial u_k} \right)_{ik}$ has full rank $m$.}, we have $r_1=m$.

i) As  \eqref{eq::linearized} is a linear system, the coefficients of the
    derivations appearing in the $\updelta_k$ only depend on analytic 
    functions of the $x_i$ and $u_j^{(k)}$ and so are constant for the derivations $\partial_{a_i}$.

    ii) If $\updelta_{k+1}=\updelta_k$, then $\hat\uptau_2^{k+1}\partial_{b_{j,0}}\in\updelta_k$, for $1\le i\le m, 0\le\ell\le k$, so that $\hat\uptau_2^{k+2}\partial_{b_{j,0}}\in\updelta_{k+1}$ and $\updelta_{k+2}\subset\updelta_k$. An easy recurrence shows then that $\updelta_{k+\ell}=\updelta_k$ for all $\ell\in\mathbb{N}$.

iii) As $\langle\partial_{b_{j,0}}|1\le j\le m\rangle\subset\updelta_k\subset\langle\partial_{b_{j,0}}|1\le j\le m\rangle+\langle\partial_{a_i}|1\le i\le n\rangle$, we need have $\updelta_{k+1}=\updelta_k$ for $k\le n$. The result is then a consequence of ii).

iv) This is a consequence of prop.~\ref{prop::accessibility}.

    v) As \eqref{eq::linearized} is a linear system, the coefficients of the
    derivations appearing in the $\updelta_k$ only depend on analytic 
    functions of the $x_i$ and $u_j^{(k)}$, \textit{i.e.} functions of the ring $
    \mathcal{O}(V)$. So finding those solutions amounts to solving a 
    linear system in the $a_i$ over the field of fractions $\mathcal{F}$ of $\mathcal{O}(V)$. So, the $\zeta_i$ are defined in the dense open set where their coefficients are analytic.

    vi) We prove that, for all $1\le\lambda\le h$, the family $\langle\zeta_{i}^{(k_i)}$, for $1\le i\le n$ and $1\le\lambda\le k\le h$, is independent. This is true by hypothesis for $\lambda=h$. Assume it is false for some smaller value and let $\lambda_0$ be the greatest value for which it is dependent. They form $\sum_{k=\lambda}^h r_k=n - \dim\updelta_{\lambda-1}$  solutions of $\updelta_{\lambda-1}\zeta=0$, so a maximal family $Z_\lambda$ of solutions when independent.
    
    Assuming that a non trivial relation exists for $\lambda_0$
    $$
    \sum_{i=1}{r_{\lambda_0}}\sum_{k=0}^{k_i-\lambda_0}c_{i,k}\zeta_i^{(k)}=0
    $$ 
    as the $\zeta_i$, for $r_{\lambda_0}<i\le r_{\lambda_0}$ are assumed to be independent of the vector space $\langle\zeta_{i}^{(k)}$, for $\lambda<\ell$, $r_{q_{\ell}-1}<i\le r_{q_{\ell}}$ and $0\le k\le q_{\ell}-k_i\rangle$, this implies the existence of a relation
    $$
    \sum_{i=1}{r_{\lambda_0-1}}\sum_{k=0}^{k_i-\lambda_0}c_{i,k}\zeta_i^{(k)}=0
    $$
    Let
    $$
    \xi:=\sum_{i=1}^{r_{\lambda_{0}-1}}c_{i,k_i-\lambda_0}\zeta_{i}^{(\lambda_0-k_i-1)},
    $$
    it is non zero, as $\lambda_0$ is maximal. This implies that
    $$\dot{\xi} \in\langle\zeta_{i}^{(k)}| 1\le i\le r_{\lambda_{0}-1},0\le k< k_i-\lambda_0-\rangle,
    $$
    so that $\xi$, which is independent of $Z_{\lambda_0+2}$ would be an extra solution of $\updelta_{\lambda_0+1}\zeta=0$, a contradiction.

    vii) By vi), the $b_{i,0}$ can be expressed as linear combinations of the $\zeta_i$ and their derivatives, with coefficient in the field generated by $\mathcal{O}(V)$. So all points in the dense open sets where these coefficient are analytic are flat.
\end{proof}

\begin{example}
    To better understand the notations, one may place the $\hat\uptau_2^{k}\partial_{b_{j,0}}$ and $\zeta_i^{(k)}$ in a Young tableau as for the following system in Brunovský form:
    \begin{equation}
    \begin{aligned}
         \dot{z}_{i,k} &=z_{i,k-1}\> \hbox{for }\> 0<k\le k_i\\
         \ddot{z}_{i,1} &=u_{i}.
    \end{aligned}
    \end{equation}

    \ytableausetup{centertableaux, boxsize=3em}
\hbox to \hsize{\begin{ytableau}
  \partial_{b_{1,0}}  &\partial_{z_{1,3}}  &\partial_{z_{1,2}} &\partial_{z_{1,1}}&\partial_{z_{1,0}}\\
  \partial_{b_{2,0}}&\partial_{z_{2,2}}&\partial_{z_{2,1}}&\partial_{z_{2,0}}\\
  \partial_{b_{3,0}}&\partial_{z_{3,2}}&\partial_{z_{3,1}}&\partial_{z_{3,0}}\\
  \partial_{b_{4,0}}&\partial_{z_{4,0}}
\end{ytableau}
\hss
\begin{ytableau}
  \zeta_1^{(4)}&\zeta_1^{(3)}&\ddot\zeta_1&\dot\zeta_1&\zeta_1\\
  \zeta_2^{(3)}&\ddot\zeta_2&\dot\zeta_2&\zeta_2\\
  \zeta_3^{(3)}&\ddot\zeta_3&\dot\zeta_3&\zeta_2\\
  \dot\zeta_4&\zeta_4
\end{ytableau} \hss}

For this system, we have $r_1=4$, $r_2=r_3=3$ and $r_4=1$, $s_1=1$, $s_2=0$, $s_3=2$ and $s_4=1$. In the same way, $p_1=4$, $p_2=p_3=4$ and $p_4=1$, $k_1=4$, $k_2=k_3=3$ and $k_4=1$. 

The system being in Brunovský from, the $\zeta_i$ may be chosen to be the $z_{i,0}$.

\end{example}

We can now sketch an algorithm in the following way. We assume that the $f_i$ are rational functions.

\begin{algorithm}\label{algo::basis}
\textbf{Input} A system $\dot{x}_i=f_i$.

\textbf{Output} A basis of the module defined by the linearized system associated to the field $\uptau_2$~\ref{eq::tau2bis}.

\textbf{Step 1.} Construct $\uptau_2$;

\textbf{Step 2.} \textbf{for} $k$ \textbf{from} $0$ to $n$, construct $\updelta_k$;

\textbf{Step 3.} Compute the $r_i$, $s_i$, $p_j$ $h$, $q_{\ell}$ and $k_i$.

\textbf{Step 4.} Let $\zeta=\emptyset$. \textbf{for} $k$ \textbf{from} $h$ to $1$, compute a maximal set of independent solutions of $\updelta_k \zeta=0$ and extract a maximal subset $Z_2$ of solutions independent of the derivatives of $Z$; Let $Z:=Z\cup Z_2$.

\textbf{Step 5.} Compute the derivatives $\zeta_{i}^{(k)}(a,b)$, for $1\le i\le r_{\lambda_{0}-1}$ and $0\le k< k_i-\lambda_0$. Solve this system to get expressions $a_i=A_i(Z)$ and $b_{j,0}=B_j(Z)$.

\textbf{Return} $Z_2$ and the parametrization $a_i=A_i(Z)$ and $b_{j,0}=B_j(Z)$.
\end{algorithm}

The main contribution to the asymptotic complexity is step~4, with a sequence of at most $n$ resolution of linear systems of size $n$, that requires at most $n^4$ operations on the fraction field of $\mathcal{0}(V)$ using naive algorithm. But, in the general case, using dense representation we cannot avoid an exponential growth of the coefficients due to successive applications of $\uptau_2$. Anyway, hand computations remain manageable for small systems.

\subsubsection{Integrability conditions}\label{susubsec::int_cond}

We need to distinguish two main cases. When $m=1$, we know by th.~\ref{th::one_control}, that any integrable symmetry is of order $0$. This situation will be investigated in section~\ref{sec::single}. 

When $m>1$, two cases may arise. First, we could be able to bound the order of $\updelta$. The only tool at our disposal for this is th.~\ref{thm::rouchon_like}, when the ideal $\mathcal{J}_D$ admits no non trivial solution. This will be considered with the first examples of section.~\ref{sec::multi}. 

When no bound is available, we need to restrict our investigations to integrable symmetries of a given order. Then, in principle, th.~\ref{th::bound_ord} allows us to restrict the integrability test to a finite number of iterations. In practice, computations would be untractable and we prefer to solve examples in the last part of sec.~\ref{sec::multi} using tricks, or changes of variables defined by flat outputs. Indeed, examples at our disposal are all flat.

\section{Examples and Computations for Single-Input Systems}\label{sec::single}

\subsection{A first example. The simplest nonflat system}\label{subsec::ex_nonflat}
In this section, we consider the following example:
$
\dot{x} = \dot{y}^2, 
$
endowed with the following Cartan field: $\uptau = \dot{y}^2 \partial_{x} + \sum_{k=0}^\infty y^{(k+1)} \partial_{y^{(k)}}$. 

Notice that this system is Lie--B\"acklund equivalent to the system:
\begin{equation}
\label{eq::explicit_form_ex1}
\left \{ \begin{array}{rcl}
\dot{x} & = & u^2 \\
\dot{y} & = & u
\end{array} \right.
\end{equation}

Let us look for a vector field $\updelta = a \partial_x + \sum_{k=0}^\infty b_k \partial_{y^{(k)}}$. According to the second commutativity condition~\eqref{eq::commutativity2}, we have $b_{k} = \uptau^k b_0$. For simplicity's sake, we will write $b$ for $b_0$. We want to show that both $a$ and $b$ depend only on the state variables $x$ and $y$. This can be done by several means: (i) direct proof, (ii) relying on theorem~\ref{thm::rouchon_like} or eventually (iii) relying on theorem~\ref{th::one_control}. 

Let us start with the direct approach.

\begin{lemma}
For this vector field to be integrable, it is necessary and sufficient that $a$ and $b = b_0$ depend on $x$ and $y$ only.     
\end{lemma}
\begin{proof}
Relying on the system equation, $\dot{x} = \dot{y}^2$, one can parametrize the diffiety with $x,y,\dot{y}, \ddot{y}, ....$. Now if $a$ or $b$ depends on $x,y, \cdots y^{(r)}$, with $r \geq 1$, then
$\updelta b = a \partial_x b + \sum_{k=0}^r b_{k} \partial_{y^{(k)}} b$. But on the other hand, $b_r = \uptau^r b$, so that it depends on $y^{(2r)}$. Recursively $\updelta^p b$ depends on $y^{(pr)}$. Therefore the sequence $b,\updelta b, \updelta^2 b, ...$ is functionally independent,
contradicting the integrability of $\updelta$. The same argument shows that $a$ also has to depend on $x$ and $y$ only. 
\end{proof}

Now let us show how theorem~\ref{thm::rouchon_like} can be applied. Consider the operator $D=A \partial_{\dot{x}} + B \partial_{\dot{y}}$. Aplying this operator twice to equation $\dot{x} = \dot{y}^2$ yields: $A = 2 \dot{y} B$ and $2B = 0$, so that $D$ actually vanish identically. Then theorem~\ref{thm::rouchon_like} and corollary~\ref{cor::k_e} imply that any integrable symmetry can only depend on the variable $x$ and $y$ at order $0$. 

Finally one can also directly apply theorem~\ref{th::one_control} to get the same conclusion.

Now let us check the conditions under which $\updelta$ commutes with $\uptau$. First we have:
$$
\begin{array}{l}
\updelta \uptau x = \updelta \dot{y}^2 = 2 \dot{y} \updelta \dot{y} = 2 \dot{y} \dot{b}  = 2 \dot{y} \uptau b = 2 \dot{y} (\dot{y}^2 \partial_x b + \dot{y} \partial_y b)\\
\uptau \updelta  x = \uptau a = \dot{y}^2 \partial_x a + \dot{y} \partial_y a
\end{array}
$$

Therefore we have the following equation:
$$
\dot{y}^2 \partial_x a + \dot{y} \partial_y a = 2 \dot{y} (\dot{y}^2 \partial_x b + \dot{y} \partial_y b)
$$

\paragraph{Direct Approach}

We proceed to the following computations:
\begin{enumerate}
    \item Derivating this equation with respect to $\dot{y}$ three times yields $\partial_x b = 0$. 
    
    \item Then derivating it twice again with respect to $\dot{y}$ yields $\partial_x a = 2 \partial_y b$. 
    
    \item Finally derivating it one time leads to $\partial_y a = 0$. 
\end{enumerate}

This implies $\partial_y \partial_x a = 2 \partial_y^2 b = 0$. Combined with $\partial_x b = 0$, we get that $b = \beta_1 y + \beta_0$. And then $a = \alpha_1 x + \alpha_0$, where $\alpha_i, \beta_i \in \R$. 

Therefore in this example one-parameter local diffeormorphism are defined by vector fields of the form:
$$
\updelta = (\alpha_1 x + \alpha_0) \partial_x + (\beta_1 y + \beta_0) \partial_y + \beta_1 \sum_{k=1}^\infty y^{(k)} \partial_{y^{(k)}}
$$

\paragraph{Parametrization Approach}

Now we shall redo the computation using the parametrization approach introduced in section~\ref{sec::parametrization}. 

For this purpose, we will consider the system under its explicit form~\eqref{eq::explicit_form_ex1}. In this context the vector fileds $\uptau$ and $\updelta$ has to be updated as follows:
$$
\begin{array}{rcl}
\uptau & = & u^2 \partial_{x} + u \partial_y + \sum_{\ell=0}^\infty u^{(\ell+1)} \partial_{u^{(\ell)}} \\
\updelta & = & a_1 \partial_{x} + a_2 \partial_y + \sum_{\ell=0}^\infty b^{(\ell+1)} \partial_{u^{(\ell)}} 
\end{array} 
$$

Then the extended system becomes:
$$
\left \{ \begin{array}{rcl}
\dot{x} & = & u^2 \\
\dot{y} & = & u \\
\dot{a}_1 & = &  2ub \\
\dot{a}_2 & = & b
\end{array} \right.
$$

Therefore the extended Cartan field $\hat{\uptau}$ is given by:
$$
\hat{\uptau} = \uptau + 2ub \partial_{a_1} + b \partial_{a_2} + \sum_{\ell=0}^\infty b^{(\ell+1)} \partial_{b^{(\ell)}}
$$

Now let us form the Lie brackets that appear in the distributions $(\updelta_p)_p$. We have:
$$
\begin{array}{rcl}
[\hat{\uptau},\partial_b] & = & [2ub\partial_{a_1},\partial_b] + [b\partial_{a_2},\partial_b] = -2u \partial_{a_1} - \partial_{a_2} \\
\hbox{[} \hat{\uptau},[\hat{\uptau},\partial_b] \hbox{]} & = & [\uptau,-2u \partial_{a_1} - \partial_{a_2}] = -2 \dot{u} \partial_{a_1} 
\end{array}
$$

Therefore the distribution $\updelta_2 = <\partial_b,-2u \partial_{a_1} - \partial_{a_2},-2 \dot{u} \partial_{a_1}> = <\partial_b, \partial_{a_1}, \partial_{a_2}>$ is a three-dimensional one, as expected according to section~\ref{sec::parametrization}, proposition~\ref{prop::algo}.

Then we look for a first integral of $\updelta_1 = <\partial_b,-2u \partial_{a_1} - \partial_{a_2}>$, which yields:
$$
\xi = a_1 - 2u a_2.
$$

Hence the parametrization follows. We have: 
$$
\begin{array}{rcl}
a_2 & = & -\frac{\hat{\uptau}{\xi}}{2 \dot{u}} \\ 
a_1 & = & \xi + 2u \frac{\hat{\uptau}{\xi}}{2 \dot{u}}
\end{array}
$$

And then $b = \hat{\uptau} a_2 = - \frac{\hat{\uptau} \hat{\uptau} \xi}{2 \dot{u}} +  \ddot{u} \frac{\hat{\uptau} \xi}{2 \dot{u}^2}$

At this stage, $\xi$ as a function of $x,y,u,\dot{u}, \cdots, u^{(r)}$ for some order $r \geq 0$ is arbitrary. Adding the constraint that the field $\updelta$ has to be integrable will impose restrictions on $\xi$.  

Indeed for the field $\updelta$ to be integrable, we need that $a_1$ and $a_2$ will depend on $x,y$ only. Since $a_2 = - \frac{\hat{\uptau}\xi}{2 \dot{u}}$, we have $\partial_{u^{(\ell)}} \xi = 0$ for $\ell \geq 1$. Therefore we have:
$$
\hat{\uptau} \xi = u^2 \partial_x \xi + u \partial_y \xi + \dot{u} \partial_u \xi.
$$

Therefore we have: $a_2 = - \frac{u^2 \partial_x \xi + u \partial_y \xi}{2 \dot{u}} - \frac{1}{2} \partial_u \xi$, which implies: $u^2 \partial_x \xi + u \partial_y \xi = 0$ (through the derivation with respect to $\dot{u}$) and then $\partial_u^2 \xi = 0$ (through the derivation with respect to $u$). 

Hence we get: $\xi = \alpha(x,y) u + \beta(x,y)$, which we can incorporate into $u^2 \partial_x \xi + u \partial_y \xi = 0$. This is equivalent to $(\partial_x \alpha) u^3 + (\partial_x \beta + \partial_y \alpha)u^2 + (\partial_y \beta) u = 0$. This reads $\partial_x \alpha = \partial_x \beta + \partial_y \alpha = \partial_y \beta = 0$. Therefore $\alpha$ depends on $y$ only and $\beta$ on $x$ only and $\partial_x \beta + \partial_y \alpha = 0$. The last equation is identically zero, which implies that $\alpha$ and $\beta$ are affine functions of $y$ and $x$ respectively. 

Thus we have: $\xi = (ay+b)u + (cx+d)$, which in turn leads to $a_2 = -\frac{1}{2}(ay+b)$ and $a_1 = (ay+b)u + (cx+d) - u(ay+b) = cx+d$. 

This is identical to the previous form obtained by the first approach.

\paragraph{Analysis}

Both approaches show that the set of compatible integrable vector fields is a distribution a real vector space of dimension $4$.

Notice that such a vector field is complete, that is the flow is defined other $\R$. In other words, the pseudogroups $G_\updelta$ are full groups. The integral curves that degenerate these groups are given by:
$$
\left \{ \begin{array}{rcl}
\gamma_1(s) & = &  \frac{k_1}{\alpha_1} e^{\alpha_1 s} - \frac{\alpha_0}{\alpha_1}\\
\gamma_2(s) & = &  \frac{k_2}{\beta_1} e^{\beta_1 s} - \frac{\beta_0}{\beta_1}\\
\gamma_{0,\ell}(s) & = & k_\ell e^{\beta_0 s}
\end{array} \right.,
$$
for some non-zeros scalar $k_1,k_2,k_\ell$ where $\ell \in \N$. 

\subsection{A second example with no nontrivial symmetries}\label{subsec::no_nontrivial}

In this section, we shall consider the following example: $\dot{x} = x^2 + y^2 + \dot{y}^2 + \dot{y}^3$. 

Notice that this system is Lie--B\"acklund equivalent to the system:
\begin{equation}
	\label{eq::explicit_form_ex2}
	\left \{ \begin{array}{rcl}
		\dot{x} & = & x^2 + y^2 + u^2 + u^3 \\
		\dot{y} & = & u
	\end{array} \right.
\end{equation}

The Cartan field can be extracted from the implicit form, which yields:
$$
\uptau = (x^2 + y^2 + \dot{y}^2 + \dot{y}^3) \partial_x + \sum_{k=0}^\infty y^{(k+1)} \partial_{y^{(k)}}
$$

Let us consider the field $\updelta = a \partial_x + \sum_{k=0}^\infty b_k \partial_{y^{(k)}}$. Here again, equations~\ref{eq::commutativity2} yields $b_k = \uptau^k b_0$. Let us simply write $b$ for $b_0$. 

\begin{lemma}
The functions $a$ and $b$ can only depend on $x$ and $y$.
\end{lemma}
\begin{proof}



The result follows from theorem~\ref{thm::rouchon_like} and corollary~\ref{cor::k_e}. It is enough to consider the operator $D = A \partial_{\dot{x}} + B \partial_{\dot{y}}$ and apply it three times to the system, which yields: $A = 2B \dot{y} + 3 B \dot{y}^2$, $2B + 6 B \dot{y} = 0$ and finally $6B = 0$, which shows that $D \equiv 0$. This, in turn, implies that $a$ and $b$ depend only on $x$ and $y$. 
\end{proof}

Now, we know the Lie bracket $[\updelta,\uptau]$ must vanish. 

Let us compute $\updelta \uptau x$:
$$
\updelta \uptau x = \updelta (x^2 + y^2 + \dot{y}^2 + \dot{y}^3) = 2x \updelta x + 2 y \updelta y + 2 \dot{y} \updelta \dot{y} + 3 \dot{y}\updelta \dot{y} = 2 x a + 2y b + 2 \dot{y}b_1 + 3 \dot{y}^2 b_1. 
$$

As already mentioned above, $b_1 = \uptau b = (x^2 + y^2 + \dot{y}^2 + \dot{y}^3) \partial_x b + \dot{y} \partial_y b$, so that we get:
$$
\updelta \uptau x = 2 x a + 2y b + (2 \dot{y} + 3 \dot{y}^2) [(x^2 + y^2 + \dot{y}^2 + \dot{y}^3) \partial_x b + \dot{y} \partial_y b]. 
$$ 

On the other hand, we have: $\uptau \updelta x = \uptau a = (x^2 + y^2 + \dot{y}^2 + \dot{y}^3) \partial_x a + \dot{y} \partial_y a$. 

Therefore $[\updelta,\uptau] = 0$ if, and only if, we have:
\begin{equation}
\label{eq::vanishing_bracket_ex2}
(x^2 + y^2 + \dot{y}^2 + \dot{y}^3) \partial_x a + \dot{y} \partial_y a = 2 x a + 2y b + (2 \dot{y} + 3 \dot{y}^2) [(x^2 + y^2 + \dot{y}^2 + \dot{y}^3) \partial_x b + \dot{y} \partial_y b].
\end{equation}

We proceed with the following computations:
\begin{enumerate}
\item Computing the derivative of equation~\eqref{eq::vanishing_bracket_ex2} with respect to $\dot{y}$ up to order 5 leads to $\partial_x b = 0$. 
Then we get the following simplified equation:
\begin{equation}
\label{eq::vanishing_bracket_ex2_1}
(x^2 + y^2 + \dot{y}^2 + \dot{y}^3) \partial_x a + \dot{y} \partial_y a = 2 x a + 2y b + (2 \dot{y} + 3 \dot{y}^2) \dot{y} \partial_y b.
	\end{equation}

\item Computing the derivative of equation~\eqref{eq::vanishing_bracket_ex2_1} with respect to $\dot{y}$ again, up to order 3, yields $\partial_x a = 3 \partial_y b$. This leads to the second following simplification:
	\begin{equation}
	\label{eq::vanishing_bracket_ex2_2}
	(x^2 + y^2 + \frac{1}{3} \dot{y}^2) \partial_x a + \dot{y} \partial_y a = 2 x a + 2y b.
\end{equation}

\item If we derive this latest equation two times with respect to $\dot{y}$ again, we get: $\partial_x a = $. Together with $\partial_x a = 3 \partial_y b$, this implies $\partial_y b = 0$. 

Therefore both partial derivatives of $b$ vanishes, so that $b$ is a constant that we will denote $\beta$ to emphasize this fact. 

The Lie bracket vanishing equation is once again simplified to this form:
\begin{equation}
	\label{eq::vanishing_bracket_ex2_3}
	\dot{y} \partial_y a = 2 x a + 2y \beta.
\end{equation} 

\item Taking the derivative of this equation with respect to $x$, yields: $\dot{y}\partial_x \partial_y a = 2a + x \partial_x a$. Since $\partial_x a = 0$ and $\partial_x \partial_y a = \partial_y \partial_x a$, we get: $a = 0$, and therefore $\beta = 0$, so that $b=0$. 
\end{enumerate}

As a consequence, we can conclude that this system has no one-parameter local symmetry at all.

\section{Examples and Computations for Multi-Input Systems}\label{sec::multi}

In this section, we consider multi-input systems, starting with the case of symmetries that depend only of the state, when systems do not satisfy the Sluis -- Rouchon criterion. In the generic case, the set of integrable symmetries is reduced to $0$. We then exhibit example with nontrivial solutions. We then consider flat systems for which symmetries of arbitrary orders do exist.

\subsection{Systems that do not satisfy the Sluis -- Rouchon criterion}\label{subsec::not_Rouchon}

Using th.~\ref{thm::rouchon_like}, we have seen that for such systems, all the integrable symmetries $\updelta$ are such that $\updelta x_i=a_i(x_1, \ldots, x_n)$. This implies that the sum of two integrable symmetries, or their Lie bracket is integrable, so that in this case integrable symmetries form a distribution of vector fields. Example~\ref{subsec::ex_nonflat} is a special case of this situation in the case of a single control.

\subsection{Generic systems}\label{subsec::generic}

We exhibit an example of system that admit no non trivial solution, and deduce that this is the case of generic algebraic system of a degree at least equal to their number of state variables.

\begin{theorem}\label{th::generic}
    i) Consider the diffiety defined by 
    \begin{equation}\label{eq::generic}
        \dot x_{m+j}=\sum_{i=1}^m (\dot
x_i)^{i+1}+\sum_{i=1}^n j^ix_i^{2},\>1\le j\le n-m.
    \end{equation} The controls here are $u_i=\dot x_i$, $1\le i\le m$. This diffiety admits no non trivial infinitesimal symmetry.

ii) Consider now a generic equation of order $1$ and degree $d\ge m$ $G(x_1', \ldots, x_n',x_1, \ldots, x_n)=0$, with coefficients $c_{\mu}$ for ${\mu\in M}$, where $M$ denotes all monomials of degree at most $d$, in the $x_i$ and $x_i'$. There exists an hypersurface $H$ in the space of coefficients $c_\mu$ such that all diffieties defined by coefficient outside $H$ admit no non trivial symmetries.
\end{theorem}
  \begin{proof} i) Applying $m$ times the operator $D:=\sum{i=1}^m A_i\partial_{\dot x_i}$ of th.~\ref{thm::rouchon_like} to any of those equations, we get $A_m^m=0$, repeating the process in sequence for all $1\le i<m$ by decreasing order, we get $A_i^i=0$, so that the only solution is $D=0$ and the theorem implies that for any integrable symmetry defined by
$\updelta x_{i}=a_{i}$, we need have $\ord\updelta=0$. 

Then, such a symmetry $\updelta$ must satisfy for all $1\le j\le n-m$:
\begin{equation}\label{eq::gen}
  \uptau a_{m+j}=\updelta \dot x_{m+1}=\updelta\left(\sum_{i=1}^m (\dot
x_i)^{1+i}+\sum_{i=1}^n j^ix_i^{2}\right)=
\sum_{i=1}^m (i+1)(\dot x_i)^{i}\uptau a_i+2\sum_{i=1}^n j^ix_i a_i.
\end{equation}
Replacing the $\uptau a_i$ by their expressions, this imply:
  
\begin{equation}\label{eq::gen2}
\begin{array}[b]{l}
\displaystyle\left(\sum_{i=1}^m (\dot x_i)^{i+1}+\sum_{i=1}^n j^ix_i^{2}\right)\partial_{x_{m+j}}
a_{m+j}+\sum_{i=1}^{m}\dot x_{i}\partial_{x_{i}}a_{m+j}=   \\
\displaystyle\hbox to 2cm{\hfill}    \sum_{i=1}^m (i+1)(\dot x_{i})^{i}\left(\left(\sum_{h=1}^m (\dot x_h)^{h+1}+\sum_{k=1}^n
x_k^{2}\right)\partial_{x_{m+1}}a_{i}+\sum_{j=1}^m \dot x_{j}\partial_{x_j}a_i\right)+2\sum_{i=1}^{n}j^ix_{i}a_{i}.
\end{array}
\end{equation}

a) Considering the $j^{\rm th}$ equation, for $i$ from $m$ to $1$, we apply to it $\partial_{\dot x_i}^{i}\partial_{\dot x_{m}}^{m+1}$,
which provides $\partial_{x_{m+j}}a_{i}=0$, then $\partial_{x_{1}}^{2}$,
which provides $\partial_{x_{m+j}}a_{m+j}=0$, then $\partial_{\dot
  x_{i}}^{i}\partial_{\dot x_{j}}$, for $(i,j)\in[1,m]^{2}$
  which gives $\partial_{x_{j}}a_{i}=0$, then
$\partial_{\dot x_{i}}$ for
$i\in[2,n]$, which gives $\partial_{x_{i}}a_{m+j}=0$.

b) We apply now $\partial_{x_{i}}$, for $i\in[1,m]$ that gives $a_{i}=0$.

c) We now have the system $\sum_{h=m+1}^n j^h x_ha_h=0$, for $m<j\le m$, that implies $a_i=0$, for $m<j\le n$.

Of this, we deduce: $\updelta=0$. 

ii) We consider the $m$ equations in the $A_i$ obtained by applying $D^i$, for $1\le i\le m$. The homogeneous resultant of these equations is a non zero polynomial $P$, depending on the $c_\mu$, $\mu\in M$. Indeed, it is non zero for the values corresponding to i), as there is then no nonzero solution.

Then, the linear equations corresponding to i) a), b) and c) have full rank for generic values of the $x_\mu$, as they have no solutions for the system considered in i). So, for values of the $c_\mu$ such that some determinant does not vanish, they admit non nonzero solution.
 \end{proof}

\subsection{A system with a priori known symmetries: $\dot{x_3} = \dot{x}_1^2 \dot{x}_2^2$}

By inspection, we see that the equation does not depend on the state, so that translation $x_i\mapsto x_i+c_i$ are symmetries. Furthermore, it is homegeneous of degree $1$ in $x_3$ and $2$ is $x_1$ and $x_2$, so that homotheties $(x_1,x_2,x_3)\mapsto (\alpha x_1, x_2, \alpha^2 x_3)$ or $(x_1,x_2,x_3)\mapsto (x_1, \alpha x_2, \alpha^2 x_3)$ are also symmetries.

As in previous examples, this system can be put in explicit form:
$$
\left \{ \begin{array}{rcl}
\dot{x}_1 & = & u_1 \\
\dot{x}_2 & = & u_2 \\
\dot{x}_3 & = & u_1^2 u_2^2 \\
\end{array} \right. 
$$

In that case, the Cartan vector field is given by:
$$
\uptau = u_1 \partial_{x_1} + u_2 \partial_{x_2} + u_1^2 u_2^2 \partial_{x_3} + \sum_{k=1}^2 \sum_{\ell \in \N} u_k^{(\ell+1)} \partial_{u_k^{(\ell)}} = \dot{x}_1 \partial_{x_1} + \dot{x}_2 \partial_{x_2} + \dot{x}_1^2 \dot{x}_2^2 \partial_{x_3} + \sum_{k=1}^2 \sum_{\ell \geq 1} x_k^{(\ell+1)} \partial_{x_k^{(\ell)}}
$$

As before, we are looking for an integrable vector $\updelta = a_1 \partial_{x_1} + a_2 \partial_{x_2} + a_3 \partial_{x_3} + \sum_{k=1}^2 \sum_{\ell \in \N} b_{k,\ell} \partial_{u_k^{(\ell)}} = a_1 \partial_{x_1} + a_2 \partial_{x_2} + a_3 \partial_{x_3} + \sum_{k=1}^2 \sum_{\ell \geq 1} b_{k,\ell} \partial_{x_k^{(\ell)}}$ that also commutes with $\uptau$. 

The commutativity system in this case is given by:
$$
\left \{ \begin{array}{rcl}
b_{1,1} & = & \uptau a_1  \\
b_{2,1} & = & \uptau a_2 \\
2 \dot{x}_1 \dot{x}_2^2 b_{1,1} + 2 \dot{x}_1^2 \dot{x}_2 b_{2,1} & = & \uptau a_3
\end{array} \right.
$$

This finally yields: 
\begin{equation}
\label{eq::commutativity_dotx3=dotx12dotx22}
\dot{a}_3 = 2 \dot{x}_1^2 \dot{x}_2 \dot{a}_2 + 2 \dot{x}_1 \dot{x}_2^2 \dot{a}_1    
\end{equation}

\begin{lemma}
The functions $a_i$ depend only on $x_1,x_2,x_3$ with no derivatives. 
\end{lemma}
\begin{proof}
This is seen again through theorem~\ref{thm::rouchon_like} and corollary~\ref{cor::k_e}. Indeed consider the operator: 
$D = A_1 \partial_{\dot{x}_1} + A_2 \partial_{\dot{x}_2} + A_3 \partial_{\dot{x}_3}$. Applying four times to the system yields 
$$
\begin{array}{ccc}
A_3 & = & 2\dot{x}_1 \dot{x}_2^2 A_1 + 2\dot{x}_1^2 \dot{x}_2 A_2 \\ 
0 & = & 2\dot{x}_2^2 A_1^2 + 8 \dot{x}_1 \dot{x_2} A_1 A_2 + 2\dot{x}_1^2 A_2^2\\
0 & = & 4\dot{x}_2 A_1^2 A_2 + 8 \dot{x}_2 A_1^2 A_2 + 8 \dot{x}_1 A_1 A_2^2  + 4\dot{x}_1 A_1 A_2^2\\
0 & = & 4 A_1^2 A_2^2 + 8 A_1^2 A_2^2 + 8 A_1^2 A_2^2 + 4 A_1^2 A_2^2\\
\end{array}
$$

This clearly implies that $A_1 A_2 = 0$, so that either $A_1$ or $A_2$ vanishes. Assuming that $A_1 = 0$, we get: $A_3 = 2\dot{x}_1^2 \dot{x}_2 A_2$ and $2 \dot{x}_1^2 A_2^2 = 0$. The latter equation implies that $A_2 = 0$, whenever $\dot{x}_1 \neq 0$. Since we look for continuous solutions, this actually means that $A_2$ identically vanishes. Therefore the same conclusion holds for $A_3$. As the situation is symmetric in $A_1$ and $A_2$, we get that $D = 0$, and by corollary~\ref{cor::k_e}, the functions $a_i$ have all zero order.  
\end{proof}







Therefore from equation~\eqref{eq::commutativity_dotx3=dotx12dotx22} we get:
$\partial_{x_1} a_3 \dot{x}_1 + \partial_{x_2} a_3 \dot{x}_2 = 2 \dot{x}_1 \dot{x}_2^2 (\partial_{x_1} a_1 \dot{x}_1 + \partial_{x_2} a_1 \dot{x}_2 + \partial_{x_3} a_1 \dot{x}_1^2 \dot{x}_2^2) + 2 \dot{x}_1^2 \dot{x}_2 (\partial_{x_1} a_2 \dot{x}_1 + \partial_{x_2} a_2 \dot{x}_2 + \partial_{x_3} a_2 \dot{x}_1^2 \dot{x}_2^2)$.

Under this reduced form, we perform now the following steps:
\begin{enumerate}
\item Applying $\partial_{\dot{x}_1^4} \partial_{\dot{x}_2^3}$ and $\partial_{\dot{x}_1^3} \partial_{\dot{x}_2^4}$, we get: $\partial_{x_3} a_1 = \partial_{x_3} a_2 = 0$. 

\item Applying $\partial_{\dot{x}_1^3} \partial_{\dot{x}_2}$ leads to $\partial_{x_2} a_1 = 0$ and similarly applying $\partial_{\dot{x}_1} \partial_{\dot{x}_2^3}$ leads to $\partial_{x_1} a_2 = 0$. 

\item Then applying $\partial_{\dot{x}_1^2} \partial_{\dot{x}_1^2}$ yields $\partial_{x_3} a_3 = 2(\partial_{x_1} a_1 + \partial_{x_2} a_2)$.

\item Then applying $\partial_{\dot{x}_1}$ and $\partial_{\dot{x}_2}$, we get $\partial_{x_1} a_3 = \partial_{x_2} a_3 = 0$. 
\end{enumerate}

Therefore, we conclude that each of the functions $a_1,a_2,a_3$ depends on a single variable as follows: $a_1(x_1)$, $a_2(x_2)$ and $a_3(x_3)$. By equation $\partial_{x_3} a_3 = 2(\partial_{x_1} a_1 + \partial_{x_2} a_2)$, the coefficients are affine functions:
$$
a_i = \alpha_i x_i + \beta_i,
$$
with $\alpha_3 = 2(\alpha_1 + \alpha_2) $.

\subsection{A one parameter group of symmeties: $\dot{x_3} = x_1^2 + x_2^2 + \dot{x}_1^3 + \dot{x}_2^3 + \dot{x}_1^2 + \dot{x}_2^2$}

We see that the system does not depend on $x_3$, so the translations $x_3\mapsto x_3+\alpha\in\mathbb{R}$ are symmetries. We will show that they are in fact the only ones. We need first the following lemma, relyin on the Sluis -- Rouchon criterion~th.~\ref{thm::rouchon_like}.

\begin{lemma}
The functions $a_i$ depends on the state vector only.
\end{lemma}
\begin{proof}
Here we consider the operator $D = A_1 \partial_{\dot{x}_1} + A_2 \partial_{\dot{x}_2} + A_3 \partial_{\dot{x}_3}$. Applying this operator to the system three times successively yields the following equations:
$$
\begin{array}{ccl}
A_3 & = & A_1 (3 \dot{x}_1^2 + 2 \dot{x}_1) + A_2 (3 \dot{x}_2^2 + 2 \dot{x}_2) \\
0 & = & A_1^2 (6 \dot{x}_1 + 2) + A_2^2 (6 \dot{x}_2 + 2) \\
0 & = & 6 A_1^3 + 6 A_2^3
\end{array}
$$

From the last equation, we get $A_2 = - A_1$, which yields, after substitution into the second equation, $A_1^2(6 \dot{x}_1 + 2 + 6 \dot{x}_2 + 2) = 0$. Therefore both $A_1$ and $A_2$ vanish. The first equation then implies that $A_3$ vanishes too and that $D = 0$, meaning that the functions $a_1,a_2,a_3$ depend on the state vector only.  

\end{proof}

The Cartan field is given by:
$$
\tau = \dot{x}_1 \partial_{x_1} + \dot{x}_2 \partial_{x_2} + \color{black} (x_1^2 + x_2^2 + \dot{x}_1^2 + \dot{x}_1^3 + \dot{x}_2^2 + \dot{x}_2^3 ) \partial_{x_3} + \sum_{k=1}^2 \sum_{\ell \geq 1} \dot{x}_k^{(\ell+1)} \partial_{\dot{x}_k^{(\ell)}}
$$

Let us consider the vector field $\updelta = a_3 \partial_{x_3} + \sum_{k=1}^2 \sum_{\ell \in\mathbb{N}} a_{k}^{(\ell)} \partial_{x_k^{(\ell)}}$. Then according to lemma~\ref{lem::commutativity} the vector field $\updelta$ commutes with $\tau$ if and only if:
\begin{equation}
\label{eq::ex3_com}
2 {x}_1 a_1 + 2 {x}_2 a_2 + \tau a_1 (3\dot{x}_1^2 +  2\dot{x}_1) + \tau a_2 (3 \dot{x}_2^2  \color{black} + 2 \dot{x}_2)  =  \tau a_3.     
\end{equation}

Equation~\eqref{eq::ex3_com} becomes 
$$
\begin{array}{l}
2 {x}_1 a_1 + 2 {x}_2 a_2 + (\dot{x}_1 \partial_{x_1}a_1 + \dot{x}_2 \partial_{x_2}a_1 + (x_1^2 + x_2^2 + \dot{x}_1^2 + \dot{x}_1^3 + \dot{x}_2^2 + \dot{x}_2^3 ) \partial_{x_3}a_1) (3\dot{x}_1^2 +  2\dot{x}_1) \\ 
+ (\dot{x}_1 \partial_{x_1}a_2 + \dot{x}_2 \partial_{x_2}a_2 + (x_1^2 + x_2^2 + \dot{x}_1^2 + \dot{x}_1^3 + \dot{x}_2^2 + \dot{x}_2^3 ) \partial_{x_3}a_2) (3\dot{x}_2^2 +  2\dot{x}_2)  \\ 
\hspace{0.5cm} =  \dot{x}_1 \partial_{x_1}a_3 + \dot{x}_2 \partial_{x_2}a_3 + (x_1^2 + x_2^2 + \dot{x}_1^2 + \dot{x}_1^3 + \dot{x}_2^2 + \dot{x}_2^3 ) \partial_{x_3}a_3     
\end{array}
$$

We proceed to the following computations:
\begin{enumerate}
\item Applying $\partial_{\dot{x}_1}^5$, we get: $\partial_{x_3}a_1 = 0$.

\item Applying $\partial_{\dot{x}_2}^5$, we get: $\partial_{x_3}a_2 = 0$.

\item Then applying $\partial_{\dot{x}_1}^3$ leads to $3 \partial_{x_1} a_1 = \partial_{x_3} a_3$,

\item Similarly applying $\partial_{\dot{x}_2}^3$ leads to $3 \partial_{x_2} a_2 = \partial_{x_3} a_3$.

\item Furthermore applying $\partial_{\dot{x}_2} \partial_{\dot{x}_1}^2$ yields $\partial_{x_2} a_1 = 0$,

\item Finally applying $\partial_{\dot{x}_1} \partial_{\dot{x}_2}^2$ yields $\partial_{x_1} a_2 = 0$.

\item Then applying $\partial_{\dot{x}_1}^2$ leads to $ \partial_{x_1} a_1 = \partial_{x_1} a_1=\partial_{x_1}a_3=0 $,

\item Similarly applying $\partial_{\dot{x}_2}^2$ leads to $ \partial_{x_2} a_2 = \partial_{x_2} a_3=0$.

\item Then applying $\partial_{{x}_1}$ leads to $ a_1 = 0 $,

\item Similarly applying $\partial_{{x}_2}$ leads to $ a_2 = 0 $.
\end{enumerate}

We then conclude that the functions have the following structure $a_1=0,a_2=0, a_3=\alpha\in\mathbb{R}$. So the group of translation : $x_3\mapsto x_3+\alpha$ is the only connected group of components. The only bigger group is obtained by adding the symmetry: $(x_1,x_2,x_3)\mapsto(x_2,x_1,x_3)$.

\subsection{A flat example, the Rouchon system: $\dot{x_3} = \dot{x_1} \dot{x_2}$}



The Cartan field is given by:
$$
\uptau = \dot{x}_1 \partial_{x_1} + \dot{x}_2 \partial_{x_2} + \dot{x}_1 \dot{x}_2   \partial_{x_3} + \sum_{k=1}^2 \sum_{\ell \geq 1} x_k^{(\ell+1)} \partial_{x_k^{(\ell)}}
 $$

The vector field $\updelta = \sum_i a_i \partial_{x_i} + \sum_{i=1}^2 \sum_{\ell \geq 1} a_{i}^{(\ell)} \partial_{x_i^{(\ell)}}$ commutes with $\uptau$ if and only if we have:

\begin{equation}
\label{eq::com_x_3=x_1 x_2}
\uptau a_3 = \dot{x}_2 \uptau a_1 + \dot{x}_1 \uptau a_2.
\end{equation}

Applying theorem~\ref{thm::rouchon_like} and corollary~\ref{cor::k_e}, we can determine the order of the functions $a_i$ with respect to $x_1,x_2$. Notice that the derivative of $x_3$ is actually a function of the derivatives of the other variables and the state.

\subsubsection{Solutions depending on the state variables only}
\label{sec::rouchon_syst_state_only}
In this section, we will compute the solutions depending on the state variables only. 

Equation~\eqref{eq::com_x_3=x_1 x_2} reads:
$$
\dot{x}_1 \partial_{x_1} a_3  + \dot{x}_2 \partial_{x_2} a_3  + \dot{x}_1 \dot{x}_2 \partial_{x_3} a_3   = \dot{x}_2 (\dot{x}_1 \partial_{x_1} a_1  + \dot{x}_2 \partial_{x_2} a1  + \dot{x}_1 \dot{x}_2 \partial_{x_3} a_1 ) + \dot{x}_1 (\dot{x}_1 \partial_{x_1} a_2  + \dot{x}_2 \partial_{x_2} a_2  + \dot{x}_1 \dot{x}_2 \partial_{x_3} a_2) 
$$

We proceed to the following sequence of computations:
\begin{enumerate}
    \item Applying $\partial_{\dot{x}_2}^2 \partial_{\dot{x}_1}$, we get: $0 = \partial_{x_3} a_1$.  Similarly applying $\partial_{\dot{x}_2} \partial_{\dot{x}_1}^2$, we get: $0 = \partial_{x_3} a_1$. Therefore the equation becomes:
$$
\partial_{x_1} a_3 \dot{x}_1 + \partial_{x_2} a_3 \dot{x}_2 + \partial_{x_3} a_3 \dot{x}_1 \dot{x}_2 = \dot{x}_2 ( \partial_{x_1} a_1 \dot{x}_1 + \partial_{x_2} a_1 \dot{x}_2 ) + \dot{x}_1 ( \partial_{x_1} a_2 \dot{x}_1 + \partial_{x_2} a_2 \dot{x}_2 ) 
$$
\item Applying $\partial_{\dot{x}_1}^2$ and $\partial_{\dot{x}_2}^2$, we get: $\partial_{x_1} a_2 = 0 = \partial_{x_2} a_1$. Therefore we have:
$$
\partial_{x_1} a_3 \dot{x}_1 + \partial_{x_2} a_3 \dot{x}_2 + \partial_{x_3} a_3 \dot{x}_1 \dot{x}_2 = \dot{x}_2 ( \partial_{x_1} a_1 \dot{x}_1 ) + \dot{x}_1 (\partial_{x_2} a_2 \dot{x}_2 ) 
$$

\item Applying $\partial_{\dot{x}_1} \partial_{\dot{x}_2}$, we get: $\partial_{x_3} a_3 = \partial_{x_1} a_1 + \partial_{x_2} a_2$. Therefore we get: $\partial_{x_1} a_3 \dot{x}_1 + \partial_{x_2} a_3 \dot{x}_2 = 0$. Applying $\partial_{\dot{x}_1}$ and then $\partial_{\dot{x}_2}$, we finally get: $\partial_{x_1} a_3 = \partial_{x_2} a_3 = 0$. 
\end{enumerate}

Therefore we have the dependency on the state variables turns out to be quite reduced, since each function depends on a single variable, as follows: $a_1(x_1), a_3(x_2), a_3(x_3)$. Since $\partial_{x_3} a_3 = \partial_{x_1} a_1 + \partial_{x_2} a_2$, we conclude that:
$$
a_i = \alpha_i x_i + \beta_i\> \text{with}\> \alpha_3=\alpha_1+\alpha_2.
$$

Each of these vector fields is complete and defines a global group of automorphisms. We have then a group of symmetries that is parametrized by $\R^5$ and corresponds to $(x_1, x_2, x_2)\mapsto (\alpha_1x_1+\beta_1,\alpha_2x_2+\beta_2,(\alpha_1+\alpha_2)x_3+\beta_3)$. 

The Lie bracket of these such fields is still a field of the same form. Therefore this family of fields does not generate higher order groups of automorphisms.

\subsubsection{Solutions depending on the state variables and their first order derivatives}

\begin{lemma}
\label{lemma::one_is_smaller}
Let $\updelta$ be an integrable symmetry. It is of iterated order $1$ iff one of the functions $a_1$ or $a_2$, say $a_{2}$, only depends on $x_2$, whereas $a_1$ and $a_3$ only depends on the state functions and $\dot x_2$.
\end{lemma}
\begin{proof}
Indeed consider the operator $D = A_1 \partial_{\dot{x}_1} + A_2 \partial_{\dot{x}_2} + A_3 \partial_{\dot{x}_3}$. Applying successively $D$ to the equation $\dot{x}_3 = \dot{x}_1 \dot{x}_2$ defining the system yields:
$$
\begin{array}{ccl}
A_3 & = & A_1 \dot{x}_2 + A_2 \dot{x}_1 \\
0 & = & 2 A_1 A_2
\end{array}
$$

Therefore either $A_1$ or $A_2$ vanishes, say $A_2=0$, which implies by corollary~\ref{cor::k_e}, that $\ord \updelta^k x_2 = 0$, for all $k\in\mathbb{N}$.

If the effective order of $\updelta$ is $1$, then $\ord a_1=1$. Assume that $a_1$ depends on $\dot x_1$, then $\ord \updelta^k x_1=k+1$, so that $\updelta$ would not be integrable. This implies that $a_1$ only depends on the state variables and $\dot x_2$. We may take $\dot x_3$ and $\dot x_1$ as controls and conclude in the same way that $a_3$ only depends on the state variables and $\dot x_2$.

As the effective order of $x_2$ is $0$, we have $\partial_{\dot x_2}\updelta^2 x_2=\partial_{\dot x_2}\updelta a_2= \partial_{\dot x_2}a_1\partial_{x_1}a_2+\partial_{\dot x_2}a_3\partial_{x_3}a_2=0$. We also have $\partial_{\dot x_2}a_1=\partial_{\ddot x_2}\dot a_1$ and $\partial_{\dot x_2}a_1=\partial_{\ddot x_2}\dot a_3=\partial_{\ddot x_2}\dot x_2 a_1$ using eq.~\eqref{eq::com_x_3=x_1 x_2}, so that 
$\partial_{\ddot x_2}\dot a_1(\partial_{x_1}a_2+\dot x_2\partial_{x_3}a_2)=0$. We assume that $a_1$ is of order $1$ and so $\partial_{\dot x_2}a_1=\partial_{\ddot x_2}\dot a_1\neq0$, which implies $\partial_{x_1}a_2+\dot x_2\partial_{x_3}a_2=0$. Applying $\partial_{\dot x_2}$, we get $\partial_{x_3}a_2=0$, and so $\partial_{x_1}a_2=0$.
\end{proof}

Using the previous lemma the following dependencies are given $a_1(x_1,x_2,x_3,\dot{x}_2)$, $a_2(x_2)$, $a_3(x_1,x_2,x_3,\dot{x}_2)$. Applying to equation~\eqref{eq::com_x_3=x_1 x_2} a sequence of computations, we can further simplify these dependencies. 

First observe that in that case, equation~\eqref{eq::com_x_3=x_1 x_2} reads:
$$
\begin{array}{l}
\partial_{x_1} a_3 \dot{x}_1 + \partial_{x_2} a_3 \dot{x}_2 + \partial_{x_3} a_3 \dot{x}_1 \dot{x}_2 + \partial_{\dot{x}_1} a_3 \ddot{x}_1 + \partial_{\dot{x}_2} a_3 \ddot{x}_2 + \partial_{\dot{x}_2} a_3 \ddot{x}_2 = \\  
\quad \dot{x}_2 ( \partial_{x_1} a_1 \dot{x}_1 + \partial_{x_2} a_1 \dot{x}_2 + \partial_{x_3} a_1 \dot{x}_1 \dot{x}_2 +  \partial_{\dot{x}_1} a_1 \ddot{x}_1 + \partial_{\dot{x}_2} a_1 \ddot{x}_2) \\
\quad + \dot{x}_1 ( \partial_{x_1} a_2 \dot{x}_1 + \partial_{x_2} a_2 \dot{x}_2 + \partial_{x_3} a_2 \dot{x}_1 \dot{x}_2 + \partial_{\dot{x}_1} a_2 \ddot{x}_1 + \partial_{\dot{x}_2} a_2 \ddot{x}_2) 
\end{array}
$$

Following the remarks above this equation can be simplified into: 
\begin{equation}\label{eq::RA}
    \partial_{x_1} a_3 \dot{x}_1 + \partial_{x_2} a_3 \dot{x}_2 + \partial_{x_3} a_3 \dot{x}_1 \dot{x}_2 + \partial_{\dot{x}_2} a_3 \ddot{x}_2 =  
\dot{x}_2 ( \partial_{x_1} a_1 \dot{x}_1 + \partial_{x_2} a_1 \dot{x}_2 + \partial_{x_3} a_1 \dot{x}_1 \dot{x}_2 + \partial_{\dot{x}_2} a_1 \ddot{x}_2)
+ \dot{x}_1\dot{x}_2 \partial_{x_2}a_2
\end{equation}

Now we proceed to the following sequence of computations:
\begin{enumerate}
    \item Applying $\partial_{\dot{x}_1}$ leads to: 
\begin{equation}\label{eq::RB}
\partial_{x_1} a_3 + \partial_{x_3} a_3 \dot{x}_2 = \partial_{x_1} a_1 \dot{x}_2 + \partial_{x_3} a_1 \dot{x}_2^2 + \partial_{x_2} a_2 \dot{x}_2.
\end{equation}

    \item Now applying $\partial_{\ddot{x}_2}$ to \eqref{eq::RA} yields: 
    \begin{equation}\label{eq::RC}
        \partial_{\dot{x}_2} a_3 = \partial_{\dot{x}_2} a_1 \dot{x}_2,
    \end{equation} which can be multiplied by $\ddot{x}_2$ and subtracted from \eqref{eq::RA} together with \eqref{eq::RB}, leading to:
    $\partial_{x_2} a_3 \dot{x}_2 = \partial_{x_2} a_1 \dot{x}_2^2$, that we may divide by $\dot{x}_2$ to get $\partial_{x_2} a_3 = \partial_{x_2} a_1 \dot{x}_2$. So we have now the following system:
    $$
    \left \{ \begin{array}{rcl}
    \partial_{x_1} a_3  & = &  -\partial_{x_3} a_3 \dot{x}_2 + \partial_{x_1} a_1 \dot{x}_2 +  \partial_{x_3} a_1 \dot{x}_2^2 + \partial_{x_2} a_2 \dot{x}_2\quad (E_1) \\
    \partial_{x_2} a_3 & = & \partial_{x_2} a_1 \dot{x}_2\quad (E_2) \\
    \partial_{\dot{x}_2} a_3 & = & \partial_{\dot{x}_2} a_1 \dot{x}_2\quad (E_3).
    \end{array} \right .
    $$

At this stage, it becomes difficult to keep going with unformal computation. The best, as these equations are linear in the derivatives of the $a_i$ is to use Gröbner bases of $D$-modules\par (see subsec.~\ref{subsubsec::GB}) over the module $\mathcal{F}[\partial_{x_i}, 1\le i\le 3, \partial_{\dot x_2}]$, where $\mathcal{F}$ is the differential field extension defined by the Rouchon system (see \textit{e.g.} \cite{Boulier2009}). We choose an ordering that compares first derivation operators, using first total order and then lexicographic order with $\partial_{x_i}>\partial_{x_j}$ if $i<j$. Then, we take derivaites of $a_i$ greater than derivatives of $a_j$ if $i>j$. This is coherent with the main derivatives appearing in the left parts of $(E_1), (E_2), (E_3)$.

    \item Applying $\partial_{x_2}$ and $\partial_{x_1}$ to respectively the first and second equation, and then subtracting the results, \textit{i.e.} computing the $S$-polynomial of $(E_1)$ and $(E_2)$ we get (after simplification):
    $$
    \partial_{x_2} \partial_{x_3} a_3 = \partial_{x_2} \partial_{x_3} a_1 \dot{x}_2 + \partial_{x_2}^2 a_2.
    $$
    Since $\partial_{x_2} a_3 =\partial_{x_2} a_1 \dot{x}_2$, we get a further simplification 
    $$
    \partial_{x_2}^2 a_2 = 0\quad (E_4).
    $$ 

    Applying $\partial_{\dot{x}_2}$ to $(E_1)$, $\partial_{x_1}$ to $(E_3)$ and computing the difference, we get the $S$-polynomial of those two equations that simplifies to: 
    $$
    \partial_{x_3} a_3 =  \partial_{x_1} a_1 + 2 \partial_{x_3} a_1 \dot{x}_2 + \partial_{x_2}a_2\quad (E_5)
    $$

    \item The reduction of the $S$-polynomial of $(E_5)$ and $(E_3)$ provides:
    $$
    \partial_{x_1} \partial_{\dot{x}_2} a_1 = - \partial_{x_3} \partial_{\dot{x}_2} a_1 \dot{x}_2 - 2 \partial_{x_3} a_1.\quad (E_6).
    $$

    Continuing in this way, the reduction of the $S$-polynomial of $(E_2)$ and $E_3)$ gives
    $$
    \partial_{x_2}a_1=0\quad (E_7);
    $$
the reduction of the $S$-polynomial of $(E_1)$ and $(E_2)$ gives
  $$
    \partial_{x_1}\partial_{x_3}a_1=0\quad (E_8);
    $$  
the reduction of the $S$-polynomial of $(E_5)$ and $(E_3)$ gives
  $$
    \partial_{x_1}\partial_{\dot x_2}a_1=-\dot x_2\partial_{x_3}\partial_{\dot x_2}a_1-2\partial_{x_3}a_1\quad (E_9);
    $$      
the reduction of the $S$-polynomial of $(E_8)$ and $(E_9)$ gives
  $$
    \dot x_2\partial_{x_3}^2\partial_{\dot x_2}a_1=\partial_{x_3}^2a_1\quad (E_{10});
    $$      
 the reduction of the $S$-polynomial of $(E_9)$ and $(E_{10})$ gives
  $$
    \partial_{x_3}^2\partial_{x_1}a_1=0\quad (E_{11}).
    $$    
All further $S$-polynomials are reduced to $0$, so that the $E_i$, $1\le i\le 10$ form a Gröbner basis $G$. One needs to solve that system to get integrable symmetries. Choosing arbitrary initial values in $\mathrm{Stair}(G)$ defines analytic solutions (see~\cite{Riquier1928}). In this case, many polynomial solutions exist for the $a_i$, for example, $a_1=\dot x_2$, $a_2=y$ and $a_3=z+\dot x_2^2/2$.

We see that computations are more demanding when the order increases and could necessitate the use of computer algebra. A study of complexity remains to be done. One may expect an exponential growth or even worse if one cannot possibly use a special structure of such systems. In the worse case, the complexity of testing if a linear PDE system admits solution is doubly exponential, as shown by Sadik~\cite{Sadik1995}.

\end{enumerate}

\section{Local classification of diffieties}\label{sec::classification}

\subsection{Classical Galois theory}

For an elementary introduction to differential Galois theory of linear differential equations, one may refer to Kaplansky~\cite[chap.~V]{Kaplansky1957} or Stewart~\cite[chap.~19]{Stewart1973}. 

The basic idea is to consider the \emph{Galois group} of a differential field extension $\mathcal{G}/\mathcal{F}$, which is the group of differential field automorphisms of $\mathcal{G}$ that leave $\mathcal{F}$ invariant. Then, the group must be solvable if the extension is Liouville, meaning that it may be obtained by a succession of elementary extensions corresponding to integrals or exponentials of integrals~\cite[th.~p.~207]{Stewart1973}. 

There exists a generalization due to Malgrange to non linear differential systems, using pseudogroups. See Casale~\cite{Casale2011}.

We will be here more concerned with the approach of Chelouah and Chitour~\cite{Chelouah2023}. They consider the control of rolling bodies and show that the system is not flat but corresponds, under some symmetry hypothesis, to a Liouville extension of a flat diffiety, that is a diffiety defined by a sequence of extensions $\mathfrak{X}_{i+1, \tau_{i+1}}/\mathfrak{X}_{i,\tau_{i}}$ with 
$$
\text{i) } \tau_{i+1}=\alpha_{i+1}\partial_{\xi_{i+1}}+\tau_i \text{ (integral case)\phantom{exponential of an $\xi_{i+1}$}}
$$
or 
$$
\text{ii) } \tau_{i+1}=\alpha_{i+1}\xi_{i+1}\partial_{\xi_{i+1}}+\tau_i \text{ (exponential of an integral case)},
$$
where $\mathfrak{X}_{i+1}$ is $\mathfrak{X}_{i+1}\times\mathbb{R}$ with coordinate function $\xi_{i+1}$ and $\alpha_{i+1}$ is a function of $\mathcal{O}(\mathfrak{X}_{i})$.

Considering algebraic systems, we may consider the Galois approach and ours, the may difference being that in Galois theory one restricts to morphisms that leave functions in a given subfield invariant, as we restrict to morphims that involve functions of a finite number of coordinates, which is always granted in classical differential Galois theory where extensions have a finite algebraic transcendance degree.

One may notice that, in many of our examples, the pseudogroups are solvable. In subsec.\ref{subsec::ex_nonflat}, it corresponds to homotheties or translations. A more complete study of the relations between these two standpoints remains to be done.

\subsection{Diffieties that are not strongly accessible}
\label{subsec::not_stringly_access}
We consider here the case of diffieties that are not strongly accessible.
Then, using prop.~\ref{prop::accessibility}, the Lie algebra $L$ generated by the $\hat\uptau^k\partial_{u_{j}}$, $1\le j\le m$, $k\in\mathbb{N}$ has dimension $n+m-d$, with $d>0$. Then we may choose new state functions $y_1, \ldots, y_{n+m}$, such that $y_1, \ldots, y_d$ are $d$ functionally independent solutions of $\delta Y(x)=0$, for $\delta\in L$. We have then $\uptau y_i=g_i(x_j|1\le j\le d)$, for $1\le i\le d$ and any vector field $\updelta=\sum_{i=1}^d a_i(y_1, \ldots, y_d)\partial y_i$, such that 
\begin{equation}\label{eq::sym_torsion}
    \sum_{j=1}^d g_j\partial_{x_j}a_i= \sum_{j=1}^d a_j\partial_{x_j}g_i
\end{equation}
is an integrable symmetry of the diffiety. In the case $\updelta x_i=0$, for all $1\le i\le d$, all vector field in this subspace is a symmetry. In the general case, using the rectification theorem, around any point in a dense open set, we may assume that $g_1=1$ and $g_i=0$, for $2\le i\le d$, so that \eqref{eq::sym_torsion} is equivalent to
\begin{equation}\label{eq::sym_torsion2}
    \partial_{x_1}a_1= 0.
\end{equation}
So, for diffieties of finite dimension, the set of integrable infinitesimal symmetries is not reduced to $\{0\}$. On may consider the restriction of the Cartan field to the subdiffiety defined by $y_i=c_i\in\mathbb{R}$, for $1\le i\le d$, that is a leaf of the foliation defined by $L$, with $\dim L=n-d$. For a choice of the $c_i$ in a dense open set and the system becomes strongly accessible, according to prop.~\ref{prop::accessibility}.

We see that a diffiety that is not strongly accessible has a non trivial set of integrable infinitesimal symmetries. A special case is that of diffieties of finite dimension, \textit{i.e.} with no control.

\subsection{Using many integrable symmetries}

If a set  $\Delta$ of two or more distinct compatible and integrable vector fields have been computed, one can consider the pseudogroup generated by the one parameter pseudogroups associated to each elements of $\Delta$. We will limit ourselves here to stating the following result.

\begin{theorem}\label{th::many_deltas}
    i) If the elements of $\Delta$ commute, then their linear combinations are integrable and they generate a $\mu:=\sharp\Delta$-parameter pseudogroup.

    ii) If $\updelta_1$ and $\updelta_2$ are integrable and $[\updelta_1,\updelta_2]$ is not, then let $\mathcal{M}$ be the (non commutative monoid) generated by $\updelta_1$ and $\updelta_2$.  We have: $\mathrm{dim}\mathcal{M}\{x_1, \ldots, x_n\}=\infty$. The pseudogroup generated by $\mathfrak{G}_{\updelta_1}$ and $\mathfrak{G}_{\updelta_2}$ is also infinite dimensional.
\end{theorem}
\begin{proof} i) We may consider the partial diffiety defined by derivations $\uptau$ and $\Delta$ and consider its characteristic set $\mathcal{A}$ for an ordering $\prec$ such that $\uptau^{k_1}\theta_1(\Delta)x_{i_1}\prec \uptau^{k_2}\theta_2(\Delta)x_{i_2}$ if $k_1<k_2$. Then the subset $\hat\Upsilon$ of $\mathrm{Stair}\mathcal{A}$ of derivatives $\theta(\Delta)x$ is finite. For any $\mu$-uple $c$ of constants, 
the derivatives $\left(\sum_{\ell}c_{\ell}\updelta_{\ell}\right)^k$, $k\in\mathbb{N}$ depend only on $\hat\Upsilon x_i$, so that they are not independent and $\sum_{\ell}c_{\ell}\updelta_{\ell}$ is integrable.
Simple computations show that $\mathfrak{g}_{\updelta_{\ell}(s)}$ denoting the one parameter pseudogroup associated to $\updelta_{\ell}$, we have 
$$
\mathfrak{g}_{\updelta_{1}}(c_1)\cdots\mathfrak{g}_{\updelta_{\mu}}(c_{\mu})=
\mathfrak{g}_{\sum_{\ell}c_{\ell}\updelta_{\ell}}(1).
$$

ii) Assume that $\mathfrak{M}:=\mathcal{M}x$ is of finite dimension $d$. Then, $[\updelta_1,\updelta_2]^k x$ may be expressed as a linear combination of elements of $\mathfrak{M}$, so that $[\updelta_1,\updelta_2]$ would be integrable. A contradiction.
\end{proof}

\subsection{Flat diffieties}\label{subsec::flatness}

These local pseudogroups of symmetries provide information on the diffiety. In particular, these pseudogroups provide a necessary condition for differential flatness. 

First, prop.~\ref{prop::structure} shows that a subset of integrable symmetries is parametrized by arbitrary analytic functions $\mathbb{R}\mapsto\mathbb{R}$, which depend on an infinite number of coefficients. So, it contains an infinite dimensional variety. One may be more precise by looking at the growth of its dimension according to its order. 

We have seen with ex.~\ref{ex::not_stable_by_addition} that the set of integrable symmetries has no simple algebraic structure. At least, cor.~\ref{cor::int_sys} shows that they can be characterized using an algebraic PDE system~\eqref{eq::integ_cond}, that completes the linear system \eqref{eq::commutativity1} that characterizes the infinitesimal symmetries\footnote{We still assume here that $u_i=\dot x_i$, so that $b_{i,k}=a_i^{(k+1)}$.}. We denote this system by $\Sigma$. It is an algebraic PDE system with coefficient in the domain $\mathcal{O}(V)$ or its quotient field. We can then freely use all results of differential algebra~\cite{Kolchin1973}.
We need then to consider a set $\Delta$ of derivatives equal to $\{\partial_{\upsilon|\ord\upsilon\le\varpi}\}$, with $\sharp\Delta=n+m\varpi$. 

It is known that the radical differential ideal $\{\Sigma\}$ generated by $\Sigma$, \textit{i.e.} the polynomials $P$ such that a power $P^d$ belongs to the differential ideal generated by $\Sigma$ is a finite intersection of prime differential ideals: $\{\Sigma\}=\bigcap_{j=1}^s\mathcal{P}_j$. We may associate to any component $\mathcal{P}$ its differential transcendence function $H(r)$, \textit{i.e.} the number of derivatives of order $r$ that belong to the characteristic set $\mathcal{A_P}$ of $\mathcal{P}$ for an \emph{orderly} ordering, that is an ordering $\prec$ such that $\theta_1x_{i_1}\prec \theta_2x_{i_2}$, for all $(i_1,i_2)\in[1,n]^2$ when $\ord\theta_1<\ord\theta_2$. 

Such a function is known to be equal to a polynomial $h(r)$ for $r$ great enough~\cite[chap.~0\S~17 lem.~16]{Kolchin1973}. It is very close to the Hilbert function of an algebraic variety~\cite{Eisenbud1995}.

\begin{proposition}\label{prop::growth}
The differential ideal describing tame symmetries $\updelta y_i=a_i$ with $\ord a_i\le\varpi$ associated to a given partition $\mathcal{Y}_j$, $1\le j\le s$, in the coordinates $y$ has degree $(m-c_s)(\varpi +1)+c_s$, with $c_j:=\sharp Y_j$.   
\end{proposition}
\begin{proof}
    The differential system satisfied by the $a_i$ is 
$$\begin{array}{rl}
    \frac{\partial a_{i_1}}{\partial x_{i_2}^{(k)}}& =0, 1\le j\le s, i_1, i_2\in Y_j,\> 0<k\le\varpi\\
    \frac{\partial a_{i_1}}{\partial x_{i_2}^{(k)}}& =0, 1\le j<\ell\le s, i_1\in Y_{j}, i_2\in Y_{\ell}, 0\le k\le\varpi.
    \end{array}
$$
    So, the differential transcendence function counting derivatives of $a_i\in Y_j$ alone up to order $r$ is that of an arbitrary function of $c_j+(\varpi+1)\sum_{\ell=1}^{j-1}c_\ell$, \textit{i.e.}
    $$
    {c_j+(\varpi+1)\sum_{\ell=1}^{j-1}c_\ell+r\choose r},
    $$
    so that 
    $$
    H(r) = \sum_{j=1}^s c_j{c_j+(\varpi+1)\sum_{\ell=1}^{j-1}c_\ell+r\choose r},
    $$
    which is of degree $c_s+(\varpi+1)\sum_{\ell=1}^{s-1}c_\ell=(m- c_s)(\varpi +1)+c_s$.
\end{proof}

The maximal degree $(m-1)(\varpi +1)+1$ is easily seen to be reached when $c_s=1$. Of this we deduce the following proposition.

\begin{proposition}\label{prop::flat_H}
    Assume that the diffiety described by system~\eqref{eq::sysbis} is flat with flat outputs of order at most $e$, the the system $\Sigma$ satisfied by functions $a_i$ or order $\varpi$ defining integrable symmetries admits a prime component $\mathcal{P}$ with a transcendence function $H$ or order at least $(m-1)(\varpi-e+1)+1$.
\end{proposition}
\begin{proof}
    It is enough to notice that in the coordinate functions associated to the flat outputs, we may consider tame integrable symmetries defined by functions of order $\varpi-e$.
\end{proof}

The main limitation of this approach is that no bound is known for the minimal order of a flat output, when $m>1$.

\section*{Conclusion}

Being given a diffiety, we are able to compute its integrable symmetries of a given order, which provides partial information on the pseudogroup of local automorphisms. This opens the hope to design more powerful tools for deciding if the set of integrable symmetries is reduced to $\{0\}$, and more generally to provide some evaluation of the growth of its dimension depending on the order. We have seen in subsec.~\ref{subsec::flatness} that this could provide a new necessary condition for differential flatness.

The structure of the pseudogroup of automorphisms of a diffiety is a difficult but promising field of investigation. A first question would be to look for new examples that could fill the gap between the generic case, with a pseudogroup reduced to identity, the diffieties  for which integrable symmetries belong to a space of small dimension and the flat diffieties, or all diffiety being flat extensions, for which the space of integrable symmetries grows quickly and the local pseudogroup of automorphisms should have a very rich structure. One may think that, in the algebraic case, the differential Cremona group, \textit{i.e.} automorphisms of a differentially transcendental extension $\mathcal{F}\langle x_1, \ldots, x_n \rangle$ already inherits the complexity of the algebraic Cremona group~\cite{Demazure1970}.

A first and possibly easier preliminary question would be to investigate if there exist non flat diffieties with an infinite dimensional space of integrable symmetries. One could start with a study of diftless systems with two controls that do no satisfy Cartan's criterion. With one control, it would be interesting to compare the complexity of computing integrable symmetries and that of known flatness criteria.

\section{Thanks}

Some preliminary results appeared in~\cite{Ollivier2010}. The authors are grateful to Bernard Malgrange for useful remarks on this work.

\end{document}